\documentclass{amsart}
\usepackage{amssymb}
\usepackage{tikz}\usetikzlibrary{decorations.markings,decorations.pathmorphing,decorations.pathreplacing,svg.path,shapes.geometric,shapes.multipart,matrix}
\usepackage{enumerate,xparse}
\usepackage{palatino}
\usepackage{euler} 
\usepackage{euscript} 
\usepackage{mathrsfs} 
\usepackage[a4paper,left=2.25cm,right=2.25cm,top=3cm,bottom=3cm]{geometry}
\usepackage{hyperref}
\allowdisplaybreaks[1]
\renewcommand\ss{\scriptstyle}
\DeclareMathOperator\Spec{Spec}
\DeclareMathOperator\rank{rank}

\DeclareMathOperator\diag{diag}
\DeclareMathOperator\Mat{Mat}
\newcommand\CC{\mathbb{C}}
\newcommand\ZZ{\mathbb{Z}}
\newcommand\NN{\mathbb{N}}

\newcommand\onto\twoheadrightarrow
\newcommand\into\hookrightarrow

\newcommand\ul\underline
\newcommand\union\cup
\newcommand\Union\bigcup
\newcommand\barX{{\overline X}}

\newcommand\iso\cong
\newcommand\dom\backslash
\newcommand\from\leftarrow

\newcommand\bs\backslash
\newcommand\dW{{\EuScript W}}
\newcommand\dE{{\EuScript E}}
\newcommand\dS{{\EuScript S}}
\newcommand\dN{{\EuScript N}}
\renewcommand\SS{{\mathscr S}}
\newcommand\tp{\dW}
\newcommand\bt{\dE}
\newcommand\wt{\mathrm{wt}}
\newcommand\tensor\otimes
\newcommand\flux\Phi

\newcommand\rem[2][]{}
\makeatletter
\newcommand{\gettikzxy}[3]{
  \tikz@scan@one@point\pgfutil@firstofone#1\relax
\pgfmathsetmacro{#2}{\the\pgf@x/\linkpatternunit}
\pgfmathsetmacro{#3}{\the\pgf@y/\linkpatternunit}
}
\tikzset{label anchor/.code={%
    \let\tikz@auto@anchor=\pgfutil@empty
    \def\tikz@anchor{#1}
  },
  label anchor/.default=center
}
\makeatother
\tikzset{arrow/.style={postaction={decorate,thick,decoration={markings,mark = at position #1 with {\arrow{>}}}}},arrow/.default=0.5}
\tikzset{invarrow/.style={postaction={decorate,thick,decoration={markings,mark = at position #1 with {\arrow{<}}}}},invarrow/.default=0.5}
\newdimen\linkpatternunit%
\newcount\linkpatternsize%
\newcount\lpsize
\newif\iflinkpatterninverted
\newif\iflinkpatterntikzstarted
\newif\iflinkpatternboxed
\newif\iflinkpatternaxis
\newif\iflinkpatternstraightlines
\newif\iflinkpatternnumbered
\newif\iflinkpatternalias
\newif\iflinkpatternnode
\newif\iflinkpatterncentered
\newcount\linkpatternfused
\pgfkeys{/linkpattern/.cd,centered/.is if=linkpatterncentered,inverted/.is if=linkpatterninverted,numbered/.is if=linkpatternnumbered,tikzstarted/.is if=linkpatterntikzstarted,straight lines/.is if=linkpatternstraightlines,boxed/.is if=linkpatternboxed,axis/.is if=linkpatternaxis,vertexcolor/.store in=\linkpatternvertexcolor,edgecolor/.store in=\linkpatternedgecolor,boxcolor/.code={\linkpatternboxedtrue\def\linkpatternboxcolor{#1}},tikzoptions/.style={every linkpattern/.append style={#1}},size/.code={\linkpatternsize=#1},numbering0/.code={\def\lpnumbering{#1}\def\linkpatternnumbering{#1}},numbering/.style={numbered,numbering0={#1}},unit/.code={\linkpatternunit=#1},height/.store in=\linkpatternheight,shape/.store in=\linkpatternshape,looseness/.store in=\linkpatternlooseness,squareness/.store in=\linkpatternsquareness,extra space/.store in=\linkpatternextraspace,width/.store in=\linkpatternwidth,alias/.is if=linkpatternalias,pos/.store in=\linkpatternpos,labeloptions/.style={labeloptionslist/.append style={#1}},labeloptionslist/.style={inner sep=2pt,font=\scriptsize,execute at begin node=$,execute at end node=$,label anchor={#1+180}},nodeon/.is if=linkpatternnode,node/.style={nodeon,nodeoptionslist/.append style={#1}},nodeoptionslist/.style={},
pipedream/.style={shape=pipedream,looseness=0,straight lines,numbering0=tangle},
tangle/.style={shape=tangle,numbering0=tangle},
every linkpattern/.style={x=\linkpatternunit,y=\linkpatternunit},
fused/.code={\linkpatternfused=#1},
vertex/.style={circle,thin,solid,draw=black,fill=\linkpatternvertexcolor,inner sep=1.5pt,draw opacity=1,transform shape},
edge/.style={very thick,solid,draw=\linkpatternedgecolor,draw opacity=1}}
\linkpatterncenteredfalse%
\linkpatterninvertedfalse%
\linkpatternnumberedfalse%
\linkpatterntikzstartedfalse%
\linkpatternboxedfalse%
\linkpatternaxistrue%
\linkpatternaliastrue%
\linkpatternstraightlinesfalse%
\def\linkpatternlooseness{0.2}
\def\linkpatternsquareness{0.35}
\def\linkpatternvertexcolor{red}%
\def\linkpatternedgecolor{blue}%
\def\linkpatternboxcolor{none}%
\def\linkpatternheight{0}
\def\linkpatternwidth{0}
\def\linkpatternshape{default}
\def\linkpatternnumbering{default}
\def\linkpatternpos{(0,0)}
\def\linkpatternextraspace{0}
\linkpatternnodefalse
\def\firstchar#1#2\empty{#1}%
\def\linkpatterndo#1#2{
\edef\param{\csname linkpattern#2\endcsname}
\edef\firstcharparam{\expandafter\firstchar\param\empty}
\expandafter\ifcat\firstcharparam a
\expandafter\ifx\csname linkpattern#1\param\endcsname\relax
\csname linkpattern#1unknown\endcsname
\else
\csname linkpattern#1\csname linkpattern#2\endcsname\endcsname
\fi
\else
\csname linkpattern#1unknown\endcsname
\fi
}%
\def\linkpatterncoordtangle{\ifnum\x>\lphalfsize\pgfmathparse{\lpsize+1-\x}\xdef\lpcoordx{\pgfmathresult}\xdef\lpcoordy{\lpheight}\xdef\lpangle{270}\else\xdef\lpcoordx{\x}\xdef\lpcoordy{-\lpheight}\xdef\lpangle{90}\fi}
\def\linkpatterncoordpipedream{\ifnum\x>\lphalfsize\pgfmathparse{\lpsize+1-\x-0.5}\xdef\lpcoordx{\pgfmathresult}\xdef\lpcoordy{0}\xdef\lpangle{270}\else\pgfmathparse{0.5-\x}\xdef\lpcoordy{\pgfmathresult}\xdef\lpcoordx{0}\xdef\lpangle{0}\fi}
\def\linkpatterncoordrectangle{
\ifnum\x>\lptqsize
\pgfmathparse{\lpsize+1-\x-0.5}\xdef\lpcoordx{\pgfmathresult}\xdef\lpcoordy{0}\xdef\lpangle{270}
\else\ifnum\x>\lphalfsize
\pgfmathparse{\x-\lptqsize-0.5}\xdef\lpcoordy{\pgfmathresult}\xdef\lpcoordx{\linkpatternwidth}\xdef\lpangle{180}
\else\ifnum\x>\linkpatternheight
\pgfmathparse{\x-\linkpatternheight-0.5}\xdef\lpcoordx{\pgfmathresult}\xdef\lpcoordy{-\linkpatternheight}\xdef\lpangle{90}
\else
\pgfmathparse{0.5-\x}\xdef\lpcoordy{\pgfmathresult}\xdef\lpcoordx{0}\xdef\lpangle{0}
\fi\fi\fi
}%
\def\linkpatternsetsizeunknown{
\global\lpsize=\linkpatternsize
\if\linkpatternheight0
\xdef\maxsep{0}
\foreach \x/\xx in \mylist%
{%
\edef\tempx{\withoutprime{\x}}
\edef\tempxx{\withoutprime{\xx}}
\pgfmathparse{max(\maxsep,abs(\tempx-\tempxx))}
\xdef\maxsep{\pgfmathresult}
}%
\pgfmathparse{0.25+0.8*\linkpatternsquareness*\maxsep}
\xdef\lpheight{\pgfmathresult}
\else
\xdef\lpheight{\linkpatternheight}
\fi
}
\newcount\tempsize
\def\linkpatternrightmostunknown{
\global\lpsize=0
\global\tempsize=0
\foreach\x/\labx in \linkpatternnumbering
{
\edef\tempx{\withoutprime{\x}}
\ifnum\lpsize<\tempx\global\lpsize=\tempx\fi
\global\advance\tempsize by 1
}
\ifnum\tempsize>\lpsize\global\lpsize=\tempsize\fi
}%
\def\linkpatternrightmostdefault{
\global\lpsize=0
\global\tempsize=0
\foreach \x/\y in \mylist
{
\edef\tempx{\withoutprime{\x}}
\ifnum\lpsize<\tempx\global\lpsize=\tempx\fi
\ifx\x\y
\global\advance\tempsize by 1
\else
\edef\tempy{\withoutprime{\y}}
\ifnum\lpsize<\tempy\global\lpsize=\tempy\fi%
\global\advance\tempsize by 2
\fi
}
\ifnum\tempsize>\lpsize\global\lpsize=\tempsize\fi
}%
\def\linkpatternrightmosttangle{
\global\lpsize=0
\global\tempsize=0
\foreach \x/\y in \mylist
{
\edef\tempx{\withoutprime{\x}}
\ifnum\lpsize<\tempx\global\lpsize=\tempx\fi
\ifx\x\y
\global\advance\tempsize by 1
\else
\edef\tempy{\withoutprime{\y}}
\ifnum\lpsize<\tempy\global\lpsize=\tempy\fi%
\global\advance\tempsize by 2
\fi
}
\global\advance\lpsize by\lpsize
\ifnum\tempsize>\lpsize\global\lpsize=\tempsize\fi
}%

\newcommand\linkpattern[2][]{
{
\pgfkeys{/linkpattern/.cd,#1}
\edef\mylist{#2}
\def\primetest##1'{}%
\def\hasaprime##1{\expandafter\primetest##1''}
\def\internalwithoutprime##1'{##1}%
\def\withoutprime##1{\if\hasaprime##1 %
\expandafter\internalwithoutprime##1\else ##1\fi}%
\iflinkpatternnumbered%
\iflinkpatterninverted
\tikzset{/linkpattern/lbl/.style n args={3}{label={[/linkpattern/labeloptionslist=-##1,##3] ##1:##2}}}%
\else%
\tikzset{/linkpattern/lbl/.style n args={3}{label={[/linkpattern/labeloptionslist=##1,##3] ##1:##2}}}%
\fi%
\else%
\tikzset{/linkpattern/lbl/.style={}}%
\fi%
\tikzifinpicture{\linkpatterntikzstartedtrue%
\begin{scope}[shift=\linkpatternpos,/linkpattern/every linkpattern]
}{%
\linkpatterntikzstartedfalse%
\iflinkpatterncentered
\begin{tikzpicture}[baseline=(current  bounding  box.center),/linkpattern/every linkpattern]%
\else
\begin{tikzpicture}[baseline=0,/linkpattern/every linkpattern]%
\fi
}%
\begin{scope}[local bounding box=link pattern box]
\iflinkpatterninverted%
\begin{scope}[yscale=-1]%
\fi%
\linkpatterndo{setsize}{shape}
\ifnum\lpsize=0
\linkpatterndo{rightmost}{numbering}
\fi
\pgfmathtruncatemacro{\lphalfsize}{\lpsize/2}
\linkpatterndo{numbering}{numbering}
\iflinkpatternboxed
\linkpatterndo{drawbox}{shape}
\else
\iflinkpatternaxis
\linkpatterndo{drawaxis}{shape}
\fi
\fi
\foreach\xx/\xlab/\opt in \lpnumbering
{
\ifx\xlab\opt\def\opt{}\fi
\if\hasaprime\xx %
\pgfmathtruncatemacro{\xx}{\lpsize+1-\withoutprime{\xx}}
\fi
\ifnum\linkpatternfused>1
\pgfmathsetmacro{\x}{0.4*(0.5+\linkpatternfused*(0.5+floor((\xx-1)/\linkpatternfused)))+0.6*\xx}
\else
\def\x{\xx}
\fi
\linkpatterndo{coord}{shape}
\iflinkpatternalias\def\xlabb{\xlab}\else\def\xlabb{\xx}\fi
\path (\lpcoordx,\lpcoordy) coordinate[/linkpattern/vertex,/linkpattern/lbl={\lpangle+180}{\xlab}{\opt},alias=v\xlabb] (v\xx) ++(\lpangle:\linkpatternunit) coordinate[alias=vv\xlabb] (vv\xx); 
}
\foreach \a/\b/\c in \mylist
{
\if\hasaprime\a %
\pgfmathtruncatemacro{\a}{\lpsize+1-\withoutprime{\a}}
\fi
\ifx\b\c\def\c{}\fi
\draw[/linkpattern/edge]
\ifx\a\b
(v\a)
\c
--
++(0,\lpheight);
\else
\pgfextra{
\if\hasaprime\b %
\pgfmathtruncatemacro{\b}{\lpsize+1-\withoutprime{\b}}
\fi
\gettikzxy{(v\a)}{\ax}{\ay}
\gettikzxy{(v\b)}{\bx}{\by}
\gettikzxy{(vv\a)}{\axx}{\ayy}
\gettikzxy{(vv\b)}{\bxx}{\byy}
\pgfmathsetmacro{\dist}{sqrt((\ax-\bx)*(\ax-\bx)+(\ay-\by)*(\ay-\by))}
\pgfmathsetmacro{\abx}{(\axx-\ax)*\dist*\linkpatternsquareness+(\bx-\ax)*\linkpatternlooseness)}
\pgfmathsetmacro{\aby}{(\ayy-\ay)*\dist*\linkpatternsquareness+(\by-\ay)*\linkpatternlooseness)}
\pgfmathsetmacro{\bax}{(\bxx-\bx)*\dist*\linkpatternsquareness+(\ax-\bx)*\linkpatternlooseness)}
\pgfmathsetmacro{\bay}{(\byy-\by)*\dist*\linkpatternsquareness+(\ay-\by)*\linkpatternlooseness)}
}
(v\a)
\c
\iflinkpatternstraightlines
\pgfextra{
\pgfmathsetmacro{\t}{((\ax-\bx)*\bay-(\ay-\by)*\bax)/(\aby*\bax-\abx*\bay)}
\pgfmathsetmacro{\abx}{\t*\abx}
\pgfmathsetmacro{\aby}{\t*\aby}
}
[rounded corners=0.2\linkpatternunit] -- ++(\abx,\aby) -- (v\b);
\else
.. controls ++(\abx,\aby) and ++(\bax,\bay) .. 
\fi
(v\b);
\fi
}
\end{scope}
\iflinkpatternnode
\node[fit=(link pattern box),/linkpattern/nodeoptionslist] {};
\fi
\iflinkpatterninverted
\end{scope}
\fi
\iflinkpatterntikzstarted
\end{scope}
\else%
\end{tikzpicture}%
\fi%
}}%
\newcommand\tanglelinkpattern[3][]{%
{
\pgfkeys{/linkpattern/.cd,#1}
\iflinkpatterninverted
\begin{tikzpicture}[/linkpattern/every linkpattern,baseline=\linkpatternunit]%
\else
\begin{tikzpicture}[/linkpattern/every linkpattern,baseline=-\linkpatternunit]%
\fi
\linkpattern[#1,tikzstarted,numbered=false]{#3}
\pgfmathtruncatemacro{\lptempsize}{2*\linkpatternsize}
\iflinkpatterninverted
\begin{scope}[yshift=0.5*\linkpatternunit]
\else
\begin{scope}[yshift=-0.5*\linkpatternunit]
\fi
\linkpattern[tangle,#1,tikzstarted,size=\lptempsize,
numbering=halftangle,
height=0.5]{#2}
\end{scope}
\end{tikzpicture}%
}}
\newcommand\schubdiag[4][]{%
\pgfkeys{/linkpattern/.cd,#1}
\iflinkpatterntikzstarted\else%
\begin{tikzpicture}[scale=0.5]
\fi%
\iflinkpatterninverted%
\begin{scope}[yscale=-1]%
\fi%
\draw (0,0) grid (#2,#3);
\edef\mylist{#4}
\foreach\y/\x/\z in \mylist
{
\ifx\x\z
\draw[decorate,decoration={zigzag,
amplitude=1pt,segment length=5pt}]
(\x-0.5,#3) -- (\x-0.5,\y-0.5) node[circle,fill=black,inner sep=2pt] {} -- (#2,\y-0.5);
\else
\node at (\x-0.5,\y-0.5) {$\z$};
\fi
}
\iflinkpatterninverted
\end{scope}
\fi
\iflinkpatterntikzstarted\else%
\end{tikzpicture}%
\fi%
}
\makeatletter
\tikzset{circle split part fill/.style  args={#1,#2}{%
 alias=tmp@name,
  postaction={%
    insert path={
     \pgfextra{%
     \pgfpointdiff{\pgfpointanchor{\pgf@node@name}{center}}%
                  {\pgfpointanchor{\pgf@node@name}{east}}%
     \pgfmathsetmacro\insiderad{\pgf@x}
      \fill[#1] (\pgf@node@name.base) ([xshift=-\pgflinewidth]\pgf@node@name.east) arc
                          (0:180:\insiderad-\pgflinewidth)--cycle;
      \fill[#2] (\pgf@node@name.base) ([xshift=\pgflinewidth]\pgf@node@name.west)  arc
                           (180:360:\insiderad-\pgflinewidth)--cycle;                    }}}}}  
 \makeatother
\tikzset{bdot/.style={circle,circle split,draw,circle split part fill={black,white},thin,inner sep=1pt}}%
\tikzset{wdot/.style={circle,circle split,draw,circle split part fill={white,black},thin,inner sep=1pt}}%
\newcommand\circlelinkpattern[2][]{
{
\pgfkeys{/linkpattern/.cd,#1}
\iflinkpatterntikzstarted\else%
\begin{tikzpicture}[/linkpattern/every linkpattern]%
\fi%
\iflinkpatterninverted%
\begin{scope}[yscale=-1]%
\fi%
\global\lpsize=\linkpatternsize
\edef\mylist{#2}
\foreach \x/\y in \mylist
{
\ifnum\x>\lpsize\global\lpsize=\x\fi
\ifnum\y>\lpsize\global\lpsize=\y\fi
}
\iflinkpatternaxis
\draw (0,0) circle (1);
\fi
\foreach\x in {1,...,\lpsize}
{
\pgfmathparse{(0.3*floor((\x-1)/\linkpatternfused)+0.7*((\x-0.5)/\linkpatternfused-0.5))*\linkpatternfused*360/\lpsize}
\coordinate[/linkpattern/vertex] (v\x) at (\pgfmathresult:1);
}
\foreach \x/\y/\z in \mylist
{
\ifx\y\z%
\draw[/linkpattern/edge] (v\x) .. controls ($0.5*(v\x)$) and  ($0.5*(v\y)$) .. (v\y);
\else
\draw[/linkpattern/edge] \z (v\x) .. controls ($0.5*(v\x)$) and  ($0.5*(v\y)$) .. (v\y);
\fi
}
\iflinkpatternnumbered%
\pgfmathparse{\lpsize/\linkpatternfused}
\global\lpsize=\pgfmathresult
\def\linkpatternnumbering{1,...,\lpsize}
\newdimen\angle
\foreach\x/\xx/\opt in \linkpatternnumbering
{
  \pgfmathsetmacro{\angle}{360/\lpsize*(\x-1)}
\ifx\xx\opt%
  \node[outer sep=1pt,anchor=180+\angle] at (\angle:1) {$\scriptstyle\xx$}; 
\else
  \node[outer sep=1pt,anchor=180+\angle,\opt] at (\angle:1) {$\scriptstyle\xx$}; 
\fi
}
\fi%
\iflinkpatterninverted%
\end{scope}
\fi%
\iflinkpatterntikzstarted\else%
\end{tikzpicture}%
\fi%
}}%
\newdimen{\loopcellsize}
\tikzset{bgplaq/.style={draw=black,fill=\linkpatternboxcolor}}
\def\plaqwest{}
\def\plaqeast{}
\def\plaqnorth{}
\def\plaqsouth{}
\pgfkeys{/linkpattern/.cd,west/.store in=\plaqwest,east/.store in=\plaqeast,north/.store in=\plaqnorth,south/.store in=\plaqsouth}
\def\plaqname{plaq}
\newcommand\plaq[2][]{
\node[bgplaq,rectangle,draw,minimum size=\loopcellsize,transform shape] (\plaqname) {};
\useasboundingbox;
\pgfkeys{/linkpattern/.cd,#1}
\ifx#2\empty\else
\begin{scope}[x=\loopcellsize,y=\loopcellsize]
\csname plaq#2\endcsname
\end{scope}\fi
}
\tikzset{loop/.code={\def\plaqname{loop-\the\pgfmatrixcurrentrow-\the\pgfmatrixcurrentcolumn}},loop/.append style={matrix,row sep={\loopcellsize,between origins},column sep={\loopcellsize,between origins}}}
\def\linkpatternboxcolor{pink!50!white}
\newcommand\plaqctr[1]{\begin{tikzpicture}[baseline={([yshift=-\the\dimexpr\fontdimen22\textfont2\relax]current  bounding  box.center)}]\plaq{#1}\end{tikzpicture}}
\newtheorem{thm}{Theorem}

\newtheorem{prop}[thm]{Proposition}
\newtheorem{cor}[thm]{Corollary}
\newtheorem{lem}[thm]{Lemma}
\theoremstyle{remark}
\newtheorem*{rmk*}{Remark}
\newtheorem*{ex*}{Example}
\numberwithin{equation}{section}

\newcommand\defn[1]{{\bf #1}}
\newcommand\junk[1]{}
\author{Allen Knutson}
\address{Allen Knutson, Cornell University, Ithaca, New York}
\email{allenk@math.cornell.edu}
\thanks{AK was supported by NSF grant DMS--2246959.}
\author{Paul Zinn-Justin}
\address{Paul Zinn-Justin, School of Mathematics and Statistics, The University of Melbourne, 
Victoria 3010, Australia}
\email{pzinn@unimelb.edu.au}
\thanks{PZJ was supported by ARC grants DP210103081 and DP240101787.}

\title{\texorpdfstring{\vspace*{-2.5cm}Hybrid pipe dreams for the lower-upper scheme}{Hybrid pipe dreams for the lower-upper scheme}}

\begin{document}
\maketitle
\begin{abstract}
  In \cite{KU-hybrid} were introduced {\em hybrid pipe dreams} interpolating
  between classic and bumpless pipe {dreams}, each hybridization giving
  a different formula for double Schubert polynomials.
  A bijective proof was given (following \cite{GaoHuang}) of the independence
  of hybridization, but only for nonequivariant Schubert polynomials.
  In this paper we further generalize to {\em hybrid generic pipe dreams},
  replacing the bijective proof of hybridization-independence with a
  Yang--Baxter-based proof that allows one to maintain equivariance.
  An additional YB-based proof establishes a divided-difference type
  recurrence for these {\em generic pipe dream polynomials}.

  These polynomials compute something richer than double Schubert
  polynomials, namely the equivariant classes of the {\em lower-upper
    varieties} introduced in \cite{Kn-uu}. We give two proofs of this:
  the easier being a proof that the recurrence relation holds on those classes,
  the more difficult being a degeneration of the lower-upper variety
  to a union of quadratic complete intersections (plus, possibly, some
  embedded components) whose individual classes match those of the
  generic pipe dreams. One new feature of the generic situation is a definition
  of the ``flux'' through an edge of the matrix; the notion of pipe dream
  itself can then be {\em derived} from the equalities among the fluxes.
\end{abstract}

\tableofcontents

\section{Introduction}\label{sec:intro}
\subsection{Prior pipe dreams}
Classic pipe dreams were introduced in \cite{BB-rcgraph,FS-Schubert} to give positive
formul\ae\ for single and double (i.e. equivariant) Schubert polynomials.
They were given a direct interpretation in \cite{KM-Schubert}:
specifically, the double Schubert polynomials were interpreted as the
equivariant cohomology classes of {\em matrix Schubert varieties,} and
each matrix Schubert variety was Gr\"obner-degenerated to a reduced union
of co\"ordinate spaces (one for each classic pipe dream).
This was unwound in \cite{K-cotransition} to a recurrence relation based
on a ``cotransition formula'' for double Schubert polynomials.

{\em Bumpless pipe dreams}\/ were introduced in \cite{LLS-BPD},
giving another formula for double Schubert polynomials, and this formula
was given a recurrence relation proof in \cite{Wei-BPD} based on the (older)
``transition formula''. The Gr\"obner geometry is more complicated here
(see \cite{KleinWeigandt}),
as the corresponding Gr\"obner degeneration is not (even generically) reduced.

One thing that is obvious is that the {\em number} of classic pipe
dreams and bumpless pipe dreams for a given permutation must coincide,
as each computes the degree of the matrix Schubert variety.
A bijection was established in \cite{GaoHuang}. A family of
intermediate objects, the {\em hybrid pipe dreams}, was introduced
in \cite{KU-hybrid}. With these came a collection of bijections of a more
local nature, each associated to a small change in hybridization.

In this paper we introduce a richer set of polynomials,
with double Schubert polynomials arising as certain leading forms thereof.
This richer story includes a recurrence relation, a geometric
interpretation, and a degeneration into simple (albeit nonlinear) components.
Some of these results were announced in \cite{artic87}.

\subsection{Hybrid generic pipe dreams}
We start with the combinatorics, leaving geometric motivation to the
next subsection.

Take $m\leq n$ natural numbers, and define $\SS_{m,n}$ to be
the set of injective maps from $\{1,\ldots,m\}$ to $\{1,\ldots,n\}$,
which we identify with maximal rank $m\times n$ partial permutation
matrices via $\pi_{ij} = 1$ if $\pi(i)=j$, zero otherwise.

Define a \defn{hybridization} as a binary sequence $\beta\in \{\tp,\bt\}^m$,
where $\tp$ and $\bt$ stand for West and East respectively.
Given a hybridization,
declare the $i^{\rm th}$ row counted from the top of a $m\times n$ grid to be ``of type'' $\beta_i$.
Define a \textbf{hybrid generic pipe dream} of type $\beta$ to be a filling of that $m\times n$ grid with the following tiles:
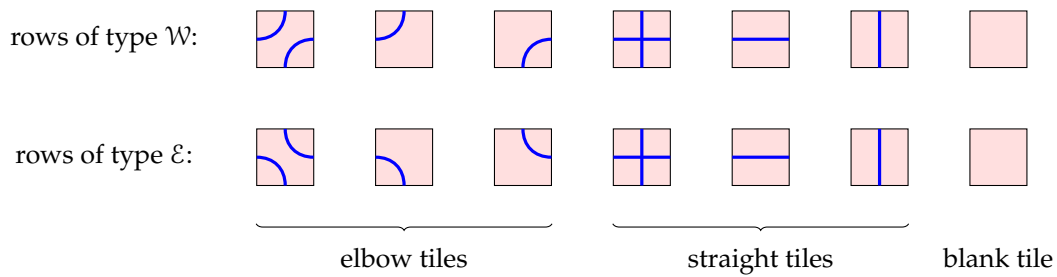
\tikzset{loopmod/.code={\def\plaqname{loop-\the\pgfmatrixcurrentrow-\the\pgfmatrixcurrentcolumn}},loop/.append style={matrix}}
\begin{center}
  \begin{tikzpicture}
\node[loop,column sep=.8cm,row sep=.8cm] {
\plaq{a} & \plaq{j} & \plaq{r} & \plaq{c} & \plaq{h} & \plaq{v} & \plaq{}
\\
\plaq{b} & \plaq{k} & \plaq{i} & \plaq{c} & \plaq{h} & \plaq{v} & \plaq{}
\\
};
\draw[decoration={brace},decorate] ([yshift=-5mm]loop-2-3.south east) -- node[below=2mm] {elbow tiles} ([yshift=-5mm]loop-2-1.south west);
\draw[decoration={brace},decorate] ([yshift=-5mm]loop-2-6.south east) -- node[below=2mm] {straight tiles} ([yshift=-5mm]loop-2-4.south west);
\draw[decoration={brace},decorate,draw=none] ([yshift=-5mm]loop-2-7.south east) -- node[below=2mm] {blank tile} ([yshift=-5mm]loop-2-7.south west);
\path (loop-1-1.west) -- ++(-2,0) node {rows of type $\tp$:};
\path (loop-2-1.west) -- ++(-2,0) node {rows of type $\bt$:};
\end{tikzpicture}
\end{center}
in such a way that
\begin{itemize}
\item 
  No pipe starts or terminates in the middle of the grid;
\item In rows of type $\tp$, a pipe enters from the West end of that row, but none exit from the East end;
\item In rows of type $\bt$, a pipe enters from the East end of that row, but none exit from the West end;
\item There are no pipes coming in from the South side.
\end{itemize}
There is no prohibition (as in other kinds of pipe dreams) on pipes
crossing twice, though this condition will come up in \S\ref{ssec:nongeneric}.

If we orient pipes in $\tp$ rows North, East, or Northeast, and orient pipes
in $\bt$ rows North, West, or Northwest, we get a consistent orientation.
In each row, no pipe exits on the East or West side (one {\em enters}
but none {\em exit}),
hence, each pipe continues upward until it comes out the North side.
Number the pipes on
the West and East sides counterclockwise, starting from the Northwest corner.
On the left below is an example with $m=3$, $n=4$, $\beta=(\tp,\bt,\tp)$:
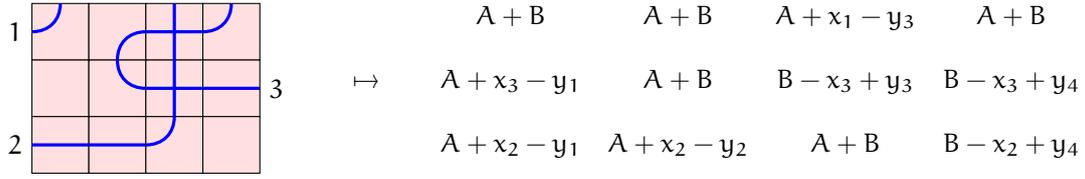
$$ 
\begin{tikzpicture}[baseline={(current bounding box.center)}]
    \node[loop]{
\plaq{j}&\plaq{r}&\plaq{c}&\plaq{j}\\
\plaq{} &\plaq{i}&\plaq{c}&\plaq{h}\\
\plaq{h}&\plaq{h}&\plaq{j}&\plaq{}\\
};
\path (loop-1-1.west) node[left] {$1$};
\path (loop-3-1.west) node[left] {$2$};
\path (loop-2-4.east) node[right] {$3$};
\end{tikzpicture}
\qquad\mapsto\qquad
\begin{matrix}
  A+B & A+B & A+x_1-y_3 & A+B \\ \\
  A+x_3-y_1 & A+B & B-x_3+y_3 & B-x_3+y_4 \\ \\
  A+x_2-y_1 & A+x_2-y_2 & A+B & B - x_2+y_4
\end{matrix}
$$ 
Let $\varphi(i)$ denote the number of the pipe entering in physical row $i$.
This $\varphi \in \SS_m$ depends only on $\beta$;
it is $\varphi = 132$ in the above example.

To a hybrid generic pipe dream we associate a polynomial in
$\ZZ[x_1,\ldots,x_m,y_1,\ldots,y_n,A,B]$, a product over locations $(i,j)$ of tiles,
of the following linear factors (as exemplified in the table above right):
\begin{itemize}
\item Elbow tiles contribute $A+B$.
\item If $\beta_i=\tp$, then blank tiles contribute $B-x_{\varphi(i)}+y_j$,
  straight tiles contribute $A+x_{\varphi(i)}-y_j$.
\item If $\beta_i=\bt$, \, then blank tiles contribute $A+x_{\varphi(i)}-y_j$,
  straight tiles contribute $B-x_{\varphi(i)}+y_j$.
\end{itemize}
In all, $(A+B)^5(A+x_1-y_3)(A+x_3-y_1)(B-x_3+y_3)(B-x_3+y_4)
(A+x_2-y_1)(A+x_2-y_2)(B-x_2+y_4)$
for the pipe dream above.

For future purposes, we denote the contribution of a tile
$\wt_{\beta_i}(\text{tile},x_{\varphi(i)}-y_j)$.

\rem[gray]{$\varphi$ determines $\beta$ except for
  last row! enumeration: $2^{m-1}$ vs $2^m$. closely mimicks the
  proof} \rem[gray]{better notation than $\varphi$, it's too cumbersome,
  what about $\beta[i]$ or $\beta:i$ or $\beta|i$ or $\beta/i$?
  $x_{\beta[i]}$, $x_{\beta:i}$, $x_{\beta|i}$, $x_{\beta/i}$}

To each pipe dream we associate also its \textbf{connectivity}
$\pi\in\SS_{m,n}$, where for $i\in [m]$ the pipe numbered $i$ exits the
North side at column $\pi(i)$.%
\junk{
  as follows: number the pipes counterclockwise on
  the West and East sides, starting from the Northwest corner, as on the
  example above, and then record the columns where they exit on the
  North side.  Also define $\varphi(i)$ to be the label (numbering) of
  the pipe entering on row $i=1,\ldots,m$; $\varphi\in\SS_m$ depends
  only on $\beta$.
}
In the example above, one has $\pi=134$.\junk{ and $\varphi=132$.}
(Again, there is no special rule concerning double crossings.)
Given a hybridization $\beta$,
define the \textbf{generic pipe dream polynomial}
$G_\pi\in \ZZ[x_1,\ldots,x_m,y_1,\ldots,y_n,A,B]$ associated to
$\pi\in\SS_{m,n}$ to be a sum over hybrid generic pipe dreams of type $\beta$
with connectivity $\pi$, where each term is the product we associated above.

Using Yang--Baxter methods coming from quantum integrable systems,
we will show
\begin{enumerate}
\item (Theorem \ref{thm:beta}) $G_\pi$ is independent of the
  hybridization $\beta$, which is why we omit it from the notation.
\item (Theorem \ref{thm:ind}) $G_{\pi r_i} = ((A+B)\partial_i - r_i) G_\pi$,
  where $r_i,\partial_i$ are the usual reflection and divided difference
  operators. This and a simple base case uniquely determine the polynomials.
\item Define a \defn{nongeneric hybrid pipe dream} by forbidding
  tiles \plaqctr{v} on rows of type $\tp$, and tiles \plaqctr{b} on rows of type $\bt$, and double crossings of pipes \cite{KU-hybrid}.
  Then the leading contribution to $G_\pi$ as $B\to\infty$ is given by the sum over
  nongeneric hybrid pipe dreams; see theorem \ref{thm:nongen} for a precise statement.
  For $m=n$, we recover the results of \cite{KU-hybrid}. In particular,
  in this limit, the relevant $(\tp,\ldots,\tp)$ (resp.\ $(\bt,\ldots,\bt)$) pipe dreams are ``classic'' (resp.\ ``bumpless'') pipe dreams,
  and the leading coefficient of $G_\pi$ is the double Schubert polynomial $S_\pi(A+x_1,\ldots,A+x_n,y_1,\ldots,y_n)$;
  see \cite{KU-hybrid} for details.
\item By left-right mirror symmetry of hybrid pipe dreams (see \S\ref{ssec:mirror}),
  a similar result holds for the $A\to\infty$ leading terms of $G_\pi$.
\end{enumerate}

\begin{ex*}
Here are the hybrid generic pipe dreams for $m=n=3$, $\pi=312$ and $\beta=(\bt,\tp,\bt)$:
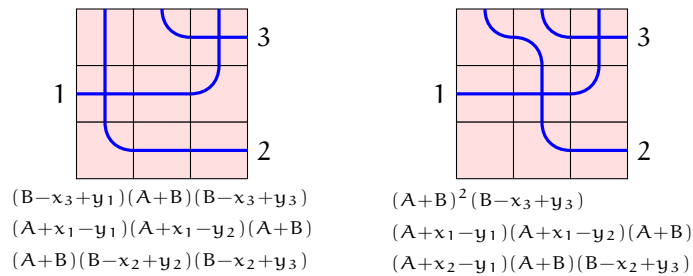
\begin{center}
\begin{tikzpicture}
\node[loop]{\plaq{v}&\plaq{i}&\plaq{c}\\\plaq{c}&\plaq{h}&\plaq{j}\\\plaq{i}&\plaq{h}&\plaq{h}\\};
\path (loop-2-1.west) node[left] {$1$};
\path (loop-3-3.east) node[right] {$2$};
\path (loop-1-3.east) node[right] {$3$};
\path (loop-3-2.south) node[below,align=left] {$\ss\left(B -x_{3}+y_{1}\right)(A+B)\left(B -x_{3}+y_{3}\right)$\\$\ss\left(A +x_{1}-y_{1}\right)\left(A +x_{1}-y_{2}\right)(A+B)$\\$\ss(A+B)\left(B -x_{2}+y_{2}\right)\left(B -x_{2}+y_{3}\right)$};
\begin{scope}[xshift=5cm]
\node[loop]{\plaq{i}&\plaq{b}&\plaq{c}\\\plaq{h}&\plaq{c}&\plaq{j}\\\plaq{}&\plaq{i}&\plaq{h}\\};
\path (loop-2-1.west) node[left] {$1$};
\path (loop-3-3.east) node[right] {$2$};
\path (loop-1-3.east) node[right] {$3$};
\path (loop-3-2.south) node[below,align=left] {$\ss(A+B)^2\left(B -x_{3}+y_{3}\right)$\\$\ss\left(A +x_{1}-y_{1}\right)\left(A +x_{1}-y_{2}\right)(A+B)$\\$\ss\left(A +x_{2}-y_{1}\right)(A+B)\left(B -x_{2}+y_{3}\right)$};
\end{scope}
\end{tikzpicture}
\end{center}

Here are the generic pipe dreams for the same $\pi$ and $\beta=(\tp,\tp,\tp)$:\\
\begin{center}
\begin{tikzpicture}
\node[loop]{\plaq{c}&\plaq{c}&\plaq{j}\\\plaq{a}&\plaq{j}&\plaq{}\\\plaq{j}&\plaq{}&\plaq{}\\};
\path (loop-1-1.west) node[left] {$1$};
\path (loop-2-1.west) node[left] {$2$};
\path (loop-3-1.west) node[left] {$3$};
\path (loop-3-2.south) node[below,align=left] {$\ss\left(A +x_{1}-y_{1}\right)\left(A +x_{1}-y_{2}\right)(A+B)$\\$\ss(A+B)^2\left(B -x_{2}+y_{3}\right)$\\$\ss(A+B)\left(B -x_{3}+y_{2}\right)\left(B -x_{3}+y_{3}\right)$};
\begin{scope}[xshift=5cm]
\node[loop]{\plaq{c}&\plaq{c}&\plaq{j}\\\plaq{j}&\plaq{v}&\plaq{}\\\plaq{h}&\plaq{j}&\plaq{}\\};
\path (loop-1-1.west) node[left] {$1$};
\path (loop-2-1.west) node[left] {$2$};
\path (loop-3-1.west) node[left] {$3$};
\path (loop-3-2.south) node[below,align=left] {$\ss\left(A +x_{1}-y_{1}\right)\left(A +x_{1}-y_{2}\right)(A+B)$\\$\ss(A+B)\left(A +x_{2}-y_{2}\right)\left(B -x_{2}+y_{3}\right)$\\$\ss\left(A +x_{3}-y_{1}\right)(A+B)\left(B -x_{3}+y_{3}\right)$};
\end{scope}
\end{tikzpicture}
\end{center}

Even though the summands are different, the sum is equal in both cases to
\[\ss
G_{312}=\left(A ^{2}+A \,B +A \,x_{2}-A \,y_{1}+B ^{2}-B \,x_{3}+B \,y_{2}+x_{2}x_{3}-x_{2}y_{1}-x_{3}y_{2}+y_{1}y_{2}\right)\left(B -x_{3}+y_{3}\right)\left(B -x_{2}+y_{3}\right)\left(A +x_{1}-y_{2}\right)\left(A +x_{1}-y_{1}\right)\left(A +B \right)^{3}
\]

In each case, only the left pipe dream is a nongeneric hybrid pipe
dream, corresponding to the leading contribution
$G_{312}\overset{B\to\infty}{\sim} (A+x_1-y_1)(A+x_1-y_2) B^7$ and
$S_{312}=(x_1-y_1)(x_1-y_2)$.
\end{ex*}

More generally, there is no hope of a bijection between the GPDs
corresponding to different choices of hybridization.  For example, for
$m=4$, $n=5$, $\pi=1253$, there are $76$, $78$, $80$ GPDs for
$\beta=(\bt,\tp,\bt,\tp)$, $(\tp,\tp,\tp,\tp)$, $(\bt,\bt,\bt,\bt)$
respectively. At best, one could introduce (following the approach of
\cite{DaojiBijective}) a notion of ``decorated generic hybrid pipe dream''
in which each blank is labeled $A$ or $B$, and seek a bijection of
those ($1475$ in this example).

\begin{rmk*}
  The reason that we call these pipe dreams ``generic'' is the following.
  In the context of quantum integrable systems, nongeneric pipe dreams
  (such as classic pipe dreams) correspond to a singular limit of the model based
  on the Yangian $\mathcal{U}_{\,\hbar}(\mathfrak{sl}_{m+1}[z])$ (and its lowest-dimensional
  fundamental representations -- for hybrid pipe dreams, one needs both of them),
  where one sends the quantum parameter $\hbar$ to infinity in an appropriate manner.
  In contrast, generic pipe dreams correspond to the case of generic $\hbar$
  (in our present language, $\hbar=A+B$).
\end{rmk*}

\subsection{Lower-upper varieties and a retrodiction of pipe dreams}
We now describe results analogous to those of \cite{KM-Schubert};
we define some varieties whose equivariant classes are the GPD
polynomials, and degenerate them to pieces corresponding to the
individual HGPDs. Perhaps the biggest difference is that the resulting
pieces are more complicated than co\"ordinate spaces. One striking
benefit of working in this ``generic'' context is that the pipe
dreams are especially clearly motivated (in \S\ref{sssec:flux}).

One possible motivation for introducing lower-upper varieties comes
from viewing them as deformed conormal matrix Schubert varieties, with
an application to the computation of Chern--Schwartz--MacPherson classes.
This point of view is considered in \cite{artic87}, but will not be
developed here.

\subsubsection{The varieties}
For $m,n \in \NN$, define the \defn{lower-upper scheme}
\[
E := \left\{ (X,Y) \in \Mat(m,n,\CC)\times \Mat(n,m,\CC):\  X Y\text{ lower triangular, } Y X\text{ upper triangular}\right\}
\]
In the $m=n$ case this scheme was introduced in \cite{Kn-uu},
but for inductive purposes we will need the rectangular case as well.
As we shall argue below, without loss of generality one may assume $m\le n$.

\rem[gray]{technically not true after degen... maybe one should worry
  about $m>n$ at some point? obviously degenerating column by column
  $m>n$ will look the same as degenerating row by row $m<n$, so we
  don't expect anything revolutionary. but we could degenerate row by
  row $m>n$ and that would presumably be messier (antiparticles appear)?}

Let $B_-$ be the group of invertible lower-triangular $m\times m$ matrices,
and $B_+$ be the group of invertible upper-triangular $n\times n$ matrices.
The group $B_-\times B_+$ acts on $\Mat(m,n,\CC)\times \Mat(n,m,\CC)$,
and on $E$, by $(b,c)\cdot (X,Y)=(b X c^{-1}, c Y b^{-1})$.

\begin{prop}\label{prop:irr}
  The irreducible components $E_\pi$ of $E$ are indexed by
  $\pi\in \SS_{m,n}$ and are given by
  \begin{align*}
    E_\pi &:= \overline{ \left\{(X,Y) \in E\colon
  (YX)_{\pi(i)\pi(i)} =     (XY)_{ii}\text{ distinct nonzero}\right\}}
\\
  &= \overline{(B_-\times B_+)\cdot \{(\pi,\pi^{T} s)\colon s\ m\times m\text{ diagonal} \}}
  \end{align*}
  In particular, the projection $(X,Y)\mapsto X$ of $E_\pi$ has
  image $\barX_\pi := \overline{B_- \pi B_+}$ the rectangular
  matrix Schubert variety. Similarly, the projection $(X,Y)\mapsto Y$
  has image $\gamma_m \barX_{\gamma_m \pi^T \gamma_n} \gamma_n$, 
  where $\gamma_m,\gamma_n$ denote the longest elements of $\SS_m$ and $\SS_n$.
\end{prop}

Inside $B_- \times B_+$ sits the torus of pairs of {\em diagonal}\/
matrices. 
We add to it two more circles that scale $X$ and $Y$ respectively.
All together they form a torus of rank $m+n+2$, denoted $T$.
We identify the equivariant cohomology ring $H^*_T(\Mat(m,n,\CC)\times \Mat(n,m,\CC))\cong H^*_T(pt)$
with $\ZZ[x_1,\ldots,x_m,y_1,\ldots,y_n,A,B]$, where the variables
are the weights of the torus action, according to
\begin{equation}\label{eq:wt}
\wt_T(X_{ij}) = A+x_i-y_j
\qquad
\wt_T(Y_{ji})  = B-x_i+y_j
\qquad
i=1,\ldots,m,\quad j=1,\ldots,n
\end{equation}
(This notation might seem to clash with the previous notation $\wt_{\tp/\bt}$ for the weights of
tiles; but in fact this is intentional, as $\wt_{\tp}(\text{straight tile},x_i-y_j)=\wt_T(X_{ij})$,
$\wt_{\tp}(\text{blank tile},x_i-y_j)=\wt_T(Y_{ji})$, and $\wt_{\tp}(\text{elbow tile},x_i-y_j)=\wt_T(X_{ij}Y_{ji})=\wt_T(X_{ij})+\wt_T(Y_{ji})$,
and similarly for $\wt_{\bt}$.)

This torus action has a two-dimensional kernel (e.g., all weights are
functions of $A+B$, $A+x_i-y_1$, $y_1-y_j$ only -- a remark which we
shall use in \S\ref{ssec:nongeneric}), but it is convenient to keep
this redundant action.

Our objective is to provide multiple combinatorial formul\ae\ for the
equivariant cohomology classes (also known as multidegrees) of the $E_\pi$.


\subsubsection{Degeneration}
Our main theorem will be
\begin{thm}\label{thm:main}\ 
  \begin{enumerate}
  \item For each hybridization $\beta$, there is a 
    degeneration of each $E_\pi$ into a union of varieties indexed by
    hybrid generic pipe dreams of type $\beta$ with connectivity $\pi$
    (possibly, though conjecturally not, with some embedded components
    of lower dimension)
    in such a way that the contribution of a pipe dream to $G_\pi$ is
    $(A+B)^m$ times the $H^*_T$ class of the corresponding variety.
  \item $G_\pi$ equals $(A+B)^m$ times the $H_T^*$ class of $E_\pi$.
    In particular it is independent of $\beta$.
  \end{enumerate}
\end{thm}

Part \textit{(2)} follows immediately from part \textit{(1)}.
We shall nevertheless give a proof of \textit{(2)} which does not rely
on \textit{(1)}, but rather on the Yang--Baxter based results of the previous section:
finding a geometric interpretation of theorem \ref{thm:ind} for the first part,
and independence of hybridization by 
theorem \ref{thm:beta}.

In the special case $\beta=(\tp,\ldots,\tp)$ and $m=n$, parts
\textit{(1)} and \textit{(2)} reduce to results announced in
\cite{artic87} (Theorems 6 and 4 respectively), so the present paper
provides proofs for these theorems as well.

\junk{maybe worth pointing out that all the issues related with BPDs
  don't occur here (everything is [generically] reduced).
  Make reference to the projection. PZJ: my attempt}
Note that components of the degeneration are generically reduced, so
the nonreducedness issues in \cite{KleinWeigandt} occur
only if one projects $(X,Y)\mapsto X$ or $(X,Y)\mapsto Y$
before degenerating.

The strange-looking bonus factor of $(A+B)^m$ can be interpreted as follows:
if we cut down $E_\pi$ by imposing $\diag(XY)=0$, the resulting equivariant
cohomology class is $G_\pi$ on the nose.
In future work we will interpret that scheme as the ``characteristic cycle''
of an open matrix Schubert variety.

\subsubsection{Fluxes}\label{sssec:flux}
We now explain how one might {\em invent} pipe dreams to index the
components of the degeneration, had they not been discovered by other means.

Denote by $V(i,j)$ the vertical edge at row $i=1,\ldots,m$ and column $j=0,\ldots,n$;
and by $H(i,j)$ the horizontal edge at row $i=0,\ldots,m$ and column $j=1,\ldots,n$.
We define \defn{flux variables} across these edges by the following:
\begin{itemize}
\item On each blank edge on the West, East, or South sides, the flux is zero.
  (The same will be true on the North side, but we don't require it
  in advance.)
\item As we move away from these outer edges, passing through the square
  in the row {\em labeled} $r$ and in {\em physical} column $j$, 
  the flux increases by $X_{rj} Y_{jr}$.
\end{itemize}
In formul\ae, 
\begin{alignat*}{2}
\Phi_{V(i,j)} &:= \sum_{\substack{j'>j\text{ row }\tp\\j'\le j\text{ row }\bt}} X_{\varphi(i)j'}Y_{j'\varphi(i)}
\qquad
&&\begin{aligned} i&=1,\ldots,m\\ j&=0,\ldots,n
\end{aligned}
\\
  \Phi_{H(i,j)} &:= \sum_{i'>i} X_{\varphi(i')j}Y_{j\varphi(i')}
\qquad&&\begin{aligned} i&=0,\ldots,m\\ j&=1,\ldots,n
\end{aligned}
\end{alignat*}
Pictorially,
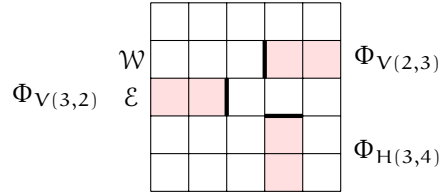
\begin{center}
  \begin{tikzpicture}[scale=0.5]
    \fill[\linkpatternboxcolor] (3,3) rectangle (5,4);
    \node at (-0.5,3.5) {$\tp$}; \node at (6.5,3.5) {$\Phi_{V(2,3)}$};
    \draw[ultra thick] (3,3) -- (3,4);
    \fill[\linkpatternboxcolor] (0,2) rectangle (2,3);
    \node at (-0.5,2.5) {$\bt$}; \node at (-2.5,2.5) {$\Phi_{V(3,2)}$};
    \draw[ultra thick] (2,2) -- (2,3);
    \fill[\linkpatternboxcolor] (3,2) rectangle (4,0); \node at (6.5,1) {$\Phi_{H(3,4)}$};
    \draw[ultra thick] (3,2) -- (4,2);    
    \draw (0,0) grid (5,5);
  \end{tikzpicture}
\end{center}
where each colored box at location $(i,j)$ stands for the summand
$X_{\varphi(i)j}Y_{j\varphi(i)}$. With such definitions it is trivial to show
that at each square, a flux conservation holds:
$$
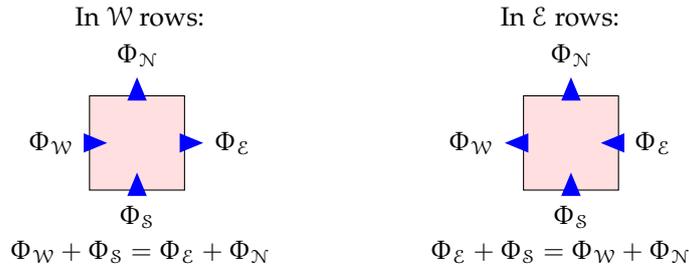
\begin{matrix}
  \text{In $\tp$ rows:} & \qquad\qquad &   \text{In $\bt$ rows:} \\
  \begin{tikzpicture}[baseline=-3pt]
  {\setlength{\loopcellsize}{1.25cm}\plaq{}}
\node[shape=isosceles triangle,shape border rotate=90,inner sep=2pt,fill=blue,label={above:$\Phi_{\dN}$}] at (plaq.north) {};
\node[shape=isosceles triangle,shape border rotate=90,inner sep=2pt,fill=blue,label={below:$\Phi_{\dS}$}] at (plaq.south) {};
\node[shape=isosceles triangle,shape border rotate=0,inner sep=2pt,fill=blue,label={right:$\Phi_{\dE}$}] at (plaq.east) {};
\node[shape=isosceles triangle,shape border rotate=0,inner sep=2pt,fill=blue,label={left:$\Phi_{\dW}$}] at (plaq.west) {};
\end{tikzpicture}
& &
\begin{tikzpicture}[baseline=-3pt]
  {\setlength{\loopcellsize}{1.25cm}\plaq{}}
\node[shape=isosceles triangle,shape border rotate=90,inner sep=2pt,fill=blue,label={above:$\Phi_\dN$}] at (plaq.north) {};
\node[shape=isosceles triangle,shape border rotate=90,inner sep=2pt,fill=blue,label={below:$\Phi_\dS$}] at (plaq.south) {};
\node[shape=isosceles triangle,shape border rotate=180,inner sep=2pt,fill=blue,label={right:$\Phi_\dE$}] at (plaq.east) {};
\node[shape=isosceles triangle,shape border rotate=180,inner sep=2pt,fill=blue,label={left:$\Phi_\dW$}] at (plaq.west) {};
\end{tikzpicture}
  \\
\Phi_\dW+\Phi_\dS=\Phi_\dE+\Phi_\dN
& &
\Phi_\dE+\Phi_\dS=\Phi_\dW+\Phi_\dN
\end{matrix}
$$

\begin{prop}[see theorem \ref{thm:mainfull}]\label{prop:flux}
  Consider the degeneration of $E_\pi$ from theorem \ref{thm:main}.
  Given a geometric component of it, defined by some ideal $I$, draw a pipe
  dream as follows: if two edges on a square have the same flux${}\bmod I$,
  and that flux is not $0\bmod I$, connect the two edges with a pipe.
  Then the result is a hybrid generic pipe dream with connectivity $\pi$,
  and this recipe
  gives a correspondence between geometric components and such HGPDs.
\end{prop}

\newcommand{\mygrid}[9]{
\begin{tikzpicture}[baseline=(current bounding box.center),yscale=1.6,xscale=.7]
  \def\hs{#1};
  \draw[lightgray,use as bounding box] (0,0) -- (\hs,0) -- (2*\hs,0) -- (2*\hs,1) -- (\hs,1) -- (0,1) -- (0,2) -- (\hs,2) -- (2*\hs,2) -- (2*\hs,1) -- (\hs,1) -- (\hs,2) -- (\hs,0) -- (0,0) -- (0,1);
  \node at (.5*\hs,2) {$#2$};
  \node at (1.5*\hs,2) {$#3$};
  \node at (0*\hs,1.5) {$#4$};
  \node at (1*\hs,1.5) {$#5$};
  \node at (2*\hs,1.5) {$0$};
  \node at (.5*\hs,1) {$#6$};
  \node at (1.5*\hs,1) {$#7$};
  \node at (0*\hs,.5) {$0$};
  \node at (1*\hs,.5) {$#8$};
  \node at (2*\hs,.5) {$#9$};
  \node at (.5*\hs,0) {$0$};
  \node at (1.5*\hs,0) {$0$};
\end{tikzpicture}
}

\begin{ex*}
We go through the details of the very simple case $m=n=2$, $\pi=12$.
The ideal is
$$ \left\langle
  x_{11}y_{21} + \ul{x_{21}y_{22}},\
  x_{11}y_{12} + \ul{x_{12}y_{22}},\
  \ul{x_{21}y_{12} - x_{12}y_{21}},\
  \ul{x_{12}y_{21} - x_{21}y_{12}}     \right\rangle
$$
which, for $\beta = (\tp,\bt)$, keeping only the underlined terms, degenerates to
$
\left\langle x_{21},\ x_{12} \right\rangle \cap
\left\langle y_{22},\ x_{21}y_{12}-x_{12}y_{21}\right\rangle 
$.

The fluxes are
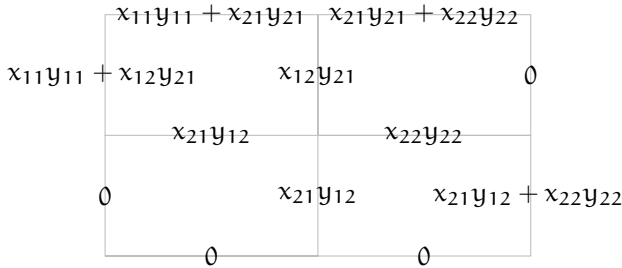
\junk{
$$
\begin{matrix}
  && x_{11}y_{11} + x_{21}y_{21} && x_{21}y_{21}+x_{22}y_{22} \\
   & x_{11}y_{11}+x_{12}y_{21} && x_{12}y_{21} && 0 \\
  && x_{21}y_{12} && x_{22}y_{22} \\
  & 0 && x_{21}y_{12} && x_{21}y_{12}+x_{22}y_{22} &  \\
  && 0 && 0
\end{matrix}
$$
}
\begin{align*}
& \mygrid 4
  {x_{11}y_{11} + x_{21}y_{21}} {x_{21}y_{21}+x_{22}y_{22}}
   {x_{11}y_{11}+x_{12}y_{21}\,} {x_{12}y_{21}}
  {x_{21}y_{12} }{ x_{22}y_{22} }
  { x_{21}y_{12} }{ x_{21}y_{12}+x_{22}y_{22}\, }
\\
\intertext{Modulo $\left\langle x_{21},\ x_{12} \right\rangle$, the fluxes become}
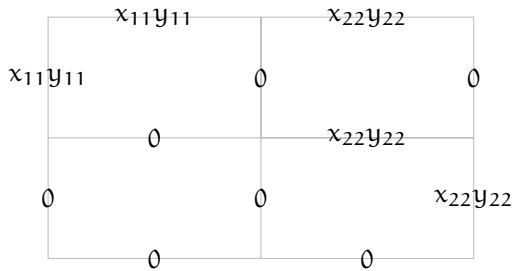
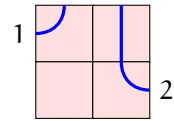
\junk{
\begin{matrix}
  && x_{11}y_{11}  && x_{22}y_{22} \\
   & x_{11}y_{11} && 0 && 0 \\
  && 0 && x_{22}y_{22} \\
  & 0 && 0 && x_{22}y_{22} &  \\
  && 0 && 0
\end{matrix}
}
&\mygrid 4
  { x_{11}y_{11}} { x_{22}y_{22} }
   { x_{11}y_{11}} 0
   0 {x_{22}y_{22} }
   0 {x_{22}y_{22} }
&& \text{producing} &&
\begin{tikzpicture}[baseline=(current bounding box.center),ampersand replacement=\&]
  \node[loop]{
    \plaq{j}\&\plaq{v}\\
    \plaq{} \&\plaq{i}\\};
\path (loop-1-1.west) node[left] {$1$};
\path (loop-2-2.east) node[right] {$2$};
\end{tikzpicture}
\\
\intertext{whereas modulo $\left\langle y_{22},\ x_{21}y_{12}-x_{12}y_{21}\right\rangle$, they become}
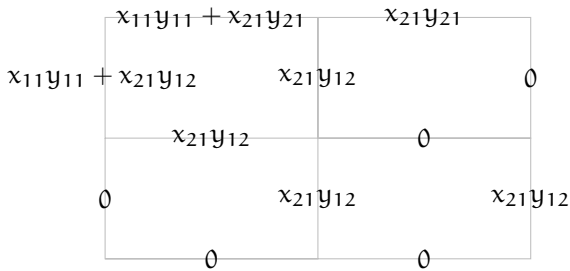
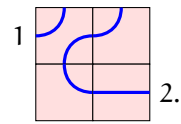
\junk{
\begin{matrix}
  && x_{11}y_{11} + x_{21}y_{21} && x_{21}y_{21} \\
   & x_{11}y_{11} + x_{21}y_{12} && x_{21}y_{12} && 0 \\
  && x_{21}y_{12} &&  \\
  & 0 && x_{21}y_{12} && x_{21}y_{12} &  \\
  && 0 && 0
\end{matrix}
}
&\mygrid 4
  {x_{11}y_{11} + x_{21}y_{21}}{ x_{21}y_{21} }
   { x_{11}y_{11} + x_{21}y_{12}\, }{ x_{21}y_{12} }
  { x_{21}y_{12} } 0
  { x_{21}y_{12} }{ x_{21}y_{12} }
&& \text{producing} &&
\begin{tikzpicture}[baseline=(current bounding box.center),ampersand replacement=\&]
  \node[loop]{
    \plaq{a}\&\plaq{j}\\
    \plaq{i}\&\plaq{h}\\};
\path (loop-1-1.west) node[left] {$1$};
\path (loop-2-2.east) node[right] {$2.$};
\end{tikzpicture}
\end{align*}
\end{ex*}

\section{The rectangular lower-upper scheme}\label{sec:lowerupper}

We now study the geometry of the rectangular lower upper-scheme
\[
  E := \left\{ (X,Y) \in \Mat(m,n,\CC)\times \Mat(n,m,\CC):\  X Y\text{ lower triangular, } Y X\text{ upper triangular}\right\}
\]
Recall that $B_-$ (resp.\ $B_+$) is the group of invertible lower-
(resp.\ upper-) triangular $m\times m$ (resp.\ $n\times n$)
matrices. We define similarly their unipotent radicals $N_-$ and
$N_+$.  Denote by $p_i$, $i=1,2$, the projection from
$E\subseteq\Mat(m,n,\CC)\times \Mat(n,m,\CC)$ to its $i^{\rm th}$ factor.
The group $B_-\times B_+$ acts on $E$ and on each factor.

\subsection{Components of the lower-upper scheme}
\begin{lem}\label{lem:dim}
  Given an $m\times n$ partial permutation matrix $\pi$, define
  $E^1_\pi:=p_1^{-1}((B_-\times B_+)\cdot\pi)$ and
  $E^2_\pi:=p_2^{-1}((B_-\times B_+)\cdot\pi^T)$.
  Then $E^1_\pi$ and $E^2_\pi$ are smooth and irreducible,
  of dimension $mn+\rank \pi$.
\end{lem}

\begin{proof}
  The proof is identical to that of \cite[lemma~1]{Kn-uu}, and we skip it.
\end{proof}

\begin{lem}\label{lem:dense}
If $\pi$ is maximal rank, then the set
\[
  E^\circ_\pi:=(N_-\times N_+)\cdot\left\{ (s\pi, \pi^T t),\ s,t\ m\times m \text{ diagonal},\ st\text{ invertible with
    distinct eigenvalues}\right\}
\]
is an open dense subset of $E^1_\pi$ and of $E^2_\pi$.
\end{lem}
\begin{proof}
It is clear that $E_\pi^\circ\subset E_\pi^1 \cap E_\pi^2$.
Computing the infinitesimal stabilizer of $(s\pi, \pi^Tt)$ under the
$N_-\times N_+$ action leads to
\[
A_- s \pi - s \pi A_+ = 0,\qquad A_+ \pi^T t - \pi^T t A_- = 0
\]
where $A_-$ is $m\times m$ strictly lower triangular, and $A_+$ is $n\times n$ strictly upper triangular.

If $\pi$ is maximal rank, $\pi\pi^T=1$, so
\[
  A_-st=A_-s\pi\pi^Tt=s \pi A_+ \pi^T t = s \pi\pi^T t A_- = st A_-
\]
hence $A_-$ commutes with $st$. Since $st$ has distinct eigenvalues,
$A_-$ is diagonal, therefore zero.  In turn this implies
\[
  \pi A_+= A_+ \pi^T=0
\]
which means that the rows and columns of $A_+$ corresponding to the
$1$s of $\pi$ are zero; removing them leaves a $(n-m)\times (n-m)$
strictly upper triangular matrix whose entries are free.

All together, the dimension of 
$(N_-\times N_+)\cdot\left\{ (s\pi, \pi^T t)\right\}$ is the dimension of the group plus the dimension of the set of
pairs of $m\times m$ diagonal matrices minus the dimension of the stabilizer,
which is
\[
m(m-1)/2+n(n-1)/2+2m-(n-m)(n-m-1)/2=m(n+1)
\]
which equals the dimension of $E^1_\pi$ and $E^2_\pi$,
giving the density statement.
\end{proof}

In particular, we conclude that $\dim E=m(n+1)$.

Note that the number of defining equations of $E$ is $m(m-1)/2+n(n-1)/2$;
this is only equal to its codimension $m(n-1)$ when $m=n$ or $m=n-1$; i.e., outside these two cases,
$E$ is not expected to be (in fact, isn't) a complete intersection.

\rem[gray]{a subtlety I forgot about: why isn't it a CI? why does the stochasticity argument for GPDs fail again?
related to modified vs not weights}

Define $E_\pi = \overline{E_\pi^1}=\overline{E_\pi^2}$ for $\pi\in \SS_{m,n}$, that is $\pi$ of maximal rank.
These are the top-dimensional components
of $E$. In appendix~\ref{app:equi}, we prove that these are in fact all the irreducible components of $E$.

We know some of the equations defining $E_\pi$:
\begin{prop}\label{prop:eqs}
  The following equations are satisfied on $E_\pi$:
  \begin{enumerate}[(i)]
  \item $XY$ is lower triangular, $YX$ is upper triangular.
  \item $\diag(YX)=\pi^T \diag(XY)\pi$.
  \item For each $i,j$, the rank of the upper-left $i\times j$ submatrix of $X$ is
    less or equal to the number of $1$s in the same submatrix of $\pi$.
  \item For each $i,j$, the rank of the lower-right $j\times i$ submatrix of $Y$ is
    less or equal to the number of $1$s in the same submatrix of $\pi^T$.
  \end{enumerate}
  Moreover, \textit{(ii')} if $\rho(i)\neq j=\pi(i)$, then the equation
  $(YX)_{jj} = (XY)_{ii}$ does {\em not}
  hold on $E_\rho$. 
\end{prop}
\begin{proof}
  \textit{(i)} is simply the definition of $E$. \textit{(ii), (iii),
    (iv)} can be easily checked on $(s\pi,\pi^Tt)$ as in
  lemma~\ref{lem:dense}, and are manifestly $N_-\times N_+$ (in fact,
  $B_-\times B_+$) invariant, therefore hold on $E^\circ_\pi$ and its
  closure $E_\pi$.

  For \textit{(ii')}, we use the second description
  in proposition \ref{prop:irr} to find a point in $E_\rho$
  on which $(YX)_{jj} \neq (XY)_{ii}$.
\end{proof}

Unfortunately, these equations are in general not sufficient to cut out $E_\pi$
(as was stated as a conjecture in \cite[\S3]{Kn-uu} for $m=n$, but
even in this case, it fails at $n=4$ for the component indexed by
$\pi=1324$). \rem[gray]{one could be more explicit...}

\subsection{Flux equations}\label{ssec:flux}
In what follows, we'll be particularly interested in equations \textit{(ii)}. Let us write them more explicitly:
\begin{equation}\label{eq:flux}
  (YX)_{jj} = \begin{cases}
    (XY)_{ii}&\text{if }j=\pi(i),\qquad i=1,\ldots,m
    \\
    0&\text{else.}
  \end{cases}
\end{equation}
Here are two results that show their importance:

\begin{lem}\label{lem:flux}\ 
  Recall the set $E_\pi^\circ$ from lemma \ref{lem:dense}.
  \begin{enumerate}
  \item One has
    \[
      E^\circ_\pi = \left\{(X,Y) \in E\colon 
        (YX)_{\pi(i)\pi(i)}=
        (XY)_{ii}\text{ distinct nonzero},\ i=1,\ldots,m\right\}
    \]
  \item
    Let $\tilde E_\pi$
    be the variety defined by equations \textit{(i)} and \textit{(ii)} of proposition~\ref{prop:eqs}.\\
    $E_\pi\subseteq \tilde E_\pi \subseteq E$ (as affine schemes),
    and $\tilde E_\pi\backslash E_\pi$ is of dimension $<\dim E$.
  \end{enumerate}
\end{lem}

\begin{proof}
  We set the stage with few linear algebra facts:
  \begin{itemize}
  \item A matrix $M$ is diagonalizable iff it satisfies a squarefree polynomial
    $p(t)$.
  \item A matrix $M$ is invertible iff it satisfies a polynomial $p(t)$
    not divisible by $t$. (Moreover the polynomials can be taken the same.)
  \end{itemize}
  Since $Y\, p(XY)\, X = q(YX)$ where $q(t) = t p(t)$,
  if $XY$ satisfies $p(t)$, then $YX$ satisfies $q(t)$.
  All put together, we learn that if $XY$ is invertible and diagonalizable,
  then $YX$ is diagonalizable. (Non-example:
  $X =  \begin{pmatrix} 0 & 1 \\ 0 & 1  \end{pmatrix}$,
  $Y =  \begin{pmatrix} 0 & 1 \\ 0 & 0  \end{pmatrix}$.)
  
  \textit{Part (1).}
  Denote the r.h.s.\ $\tilde E_\pi^\circ$. It is clear that $E_\pi^{\circ}\subseteq \tilde E_\pi^\circ$.
  Conversely, consider $(X,Y)\in \tilde E_\pi^\circ$.
  By the above argument,
  we find that both $XY$ and $YX$ are diagonalizable; more precisely,
  there exists $U_-\in B_-$ (resp.\ $U_+\in B_+$) such that $U_-XYU_-^{-1}$
  (resp.\ $U_+YXU_+^{-1}$)
  is diagonal with same diagonal entries as $XY$ (resp.\ $YX$).
  Now consider $(X',Y')=(U_-,U_+)\cdot (X,Y)$. 
  Writing $X'Y'X'$ in two different ways leads to
  \[
    (X'Y'X')_{ij}=(XY)_{ii} X'_{ij} = X'_{ij} (YX)_{jj}
  \]
  and since the $(XY)_{ii}=(YX)_{\pi(i)\pi(i)}$ are distinct nonzero,
  we conclude that $X'_{ij}=0$ unless $\pi(i)=j$.
  Similarly,
  \[
    (Y'XY')_{ji}=(YX)_{jj} Y'_{ji} = Y'_{ji} (XY)_{ii}
  \]
  so $Y'_{ji}=0$ unless $\pi(i)=j$. That $(X,Y)\in E_\pi^\circ$ follows
  immediately.

  \textit{Part (2).} It is clear by definition that $E_\pi \subseteq \tilde E_\pi \subseteq E$.
  Next, consider the decomposition
  \[
    \tilde E_\pi \backslash E_\pi^1 = \bigsqcup_{\rho\ne\pi} (\tilde E_\pi \cap E_\rho^1)
  \]
  If $\rho$ is not full rank, $\dim E_\rho^1<\dim E$. If $\rho$ is full rank, with $\rho\ne\pi$,
  note that $\tilde E_\pi \cap E^\circ_\rho = \varnothing$,
  because the equations \textit{(ii)} of proposition~\ref{prop:eqs} for $\tilde E_\pi$ contradict the distinct nonzero eigenvalues of $XY$ in
  part (1) of the lemma applied to $E^\circ_\rho$. Therefore $\dim(\tilde E_\pi\cap E_\rho^1)\le \dim (E_\rho^1\backslash E^\circ_\rho)< \dim E$.
\end{proof}

Putting together lemma~\ref{lem:dense} and the first part of lemma~\ref{lem:flux} leads to proposition~\ref{prop:irr}.

\begin{rmk*}
  We call \eqref{eq:flux} ``flux equations''. To justify this name, we
  recall the definition of fluxes in \S\ref{sssec:flux}.
  $(XY)_{ii}$ is the flux through the leftmost or rightmost vertical edge of the row where pipe labeled $i$
  enters a pipe dream, depending on whether that row is of type $\tp$ or $\bt$;
  so it should be interpreted as ``the flux of pipe numbered $i$''.
  Similarly, $(YX)_{jj}$ is the flux through the North
  side of a pipe dream at column $j$, so that writing $(YX)_{jj}=(XY)_{ii}$ for $\pi(i)=j$ simply expresses the connectivity
  of any pipe dream derived from $E_\pi$. This statement will be made precise in \S\ref{sec:degen}.
\end{rmk*}

\subsection{Generic reducedness}

For future use, we now show:

\begin{lem}\label{lem:genred}
$E$ is generically reduced.
\end{lem}

\begin{proof}
  This is a Zariski tangent space calculation.
  Consider the point $(s\pi,\pi^Tt)\in E_\pi^\circ$ and expand at
  first order the equations of $E$:
\[
T_{(s\pi,\pi^T t)}E=\{(X,Y)\colon s\pi Y+X\pi^Tt \text{ lower triangular, }\pi^Tt X+Ys\pi\text{ upper triangular}\}
\]
So we have equations of the form
\begin{align*}
s_{i}Y_{\pi(i)i'}+t_{i'}X_{i\pi(i')}&=0 && i<i'
\\
t_{\pi^{-1}(j)}X_{\pi^{-1}(j)j'}+s_{\pi^{-1}(j')}Y_{j\pi^{-1}(j')}&=0 && j>j'
\end{align*}
where in the second line it is understood that if $j$ or $j'$ is not in the image of $\pi$, then the corresponding term is missing.

For $s$ and $t$ as in lemma~\ref{lem:dense}, the nonzero equations above are linearly independent
(the only nontrivial case is pairs $i<i'$, $j=\pi(i)>j'=\pi(i')$, in which equations of the first and second line involve the same
variables, and we have to use $s_it_i\ne s_{i'}t_{i'}$).
There are $m(m-1)/2$ of the first kind, and $n(n-1)/2-(n-m-1)(n-m-1)/2$ of the second kind (all pairs $(j,j')$ except when neither $j$ nor $j'$ is in the image of $\pi$).
Summing up leads to $m(n-1)$ which is the codimension of $E$.
\end{proof}

In the $m=n$ case, generic reducedness was proved much the same way in
\cite[theorem 2]{Kn-uu}, where it implied reducedness thanks to $E$
being a complete intersection.

\subsection{An involution}\label{ssec:invo}
Finally, we point out the following simple observation, which will be useful below.
Recall that $\gamma_m$ and $\gamma_n$ are the longest elements of $\SS_m$ and $\SS_n$.

\begin{lem}
The involution $\iota:(X,Y)\mapsto (\gamma_m Y^T \gamma_n,\gamma_n X^T \gamma_m)$ preserves $E$,
and sends $E_\pi$ to $E_{\gamma_n \pi \gamma_m}$.
\end{lem}
\begin{proof}
One first checks that the defining equations of $E$ are invariant under $\iota$.
This implies that the components $E_\pi$ are permuted by the involution $\iota$.
In order to figure out the permutation,
we use proposition~\ref{prop:eqs} \textit{(ii)} and \textit{(ii')}.
Denote $(X',Y')=\iota(X,Y)$. Then
\begin{align*}
(Y'X')_{jj}=(\gamma_n X^T Y^T \gamma_n)_{jj} = (YX)_{\gamma_n(j)\gamma_n(j)} &=
\begin{cases}
(XY)_{ii} & \gamma_n(j)=\pi(i)\\
0&\text{else}
\end{cases}
\\
&=
\begin{cases}
(XY)_{\gamma_m(i)\gamma_m(i)} & \gamma_n(j)=\pi(\gamma_m(i))\\
0&\text{else}
\end{cases}
\\
&=
\begin{cases}
(X'Y')_{ii}& j=\gamma_n\pi\gamma_m(i)
\\
0&\text{else}
\end{cases}
\end{align*}
\end{proof}
This $\iota$ is not $T$-equivariant but rather $T$-normalizing (the
combined actions form a semi-direct product $T \rtimes \ZZ_2$, not a
direct product). By keeping tracks of the weights, one finds:

\begin{cor} The following equality holds:
  \[
    [E_\pi]
    =
    [E_{\gamma_n\pi\gamma_m}]|_{A\leftrightarrow B,x_i\leftrightarrow\, -x_{\gamma_m(i)},y_j\leftrightarrow\, -y_{\gamma_n(j)}}
\]
\end{cor}

\begin{rmk*}
  One can also consider the mapping $(X,Y)\to (X^T,Y^T)$, which sends
  the $m\times n$ lower-upper scheme to the $n\times m$ lower-upper
  scheme, so is only a symmetry of $E$ in the case $m=n$ (then sending
  $E_\pi$ to $E_{\pi^{-1}}$). This also justifies that we may assume
  without loss of generality $m\le n$.
\end{rmk*}

\section{Pipe dreams and Yang--Baxter equations}

Recall that our pipes are oriented upwards, i.e. they enter on the
West side in $\tp$ rows, on the East side in $\bt$ rows, and exit on the North.
For the purposes of this section, we'll often drop indices on the row and column parameters, 
so that to each square are attached two parameters say $x$ and $y$;
and we'll write them next to that square or row/column of squares.
We also allow ourselves to deform rows slightly in such a way that the squares become parallelograms, as will become clear on the pictures
below. Finally, to distinguish squares (or parallelograms) on rows of type $\beta_i\in \{\tp,\bt\}$,
we'll use arrows that show the direction of propagation of the pipes:
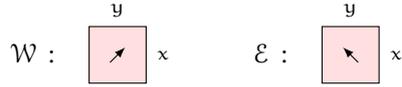
\[
\tp\,:\quad \tikz[baseline=-3pt]{\plaq{}\draw[-latex] (-.1,-.1) -- (.1,.1); \node at (.6,0) {$\ss x$}; \node at (0,.6) {$\ss y$};}
\qquad\qquad
\bt\,:\quad \tikz[baseline=-3pt]{\plaq{}\draw[-latex] (.1,-.1) -- (-.1,.1); \node at (.6,0) {$\ss x$}; \node at (0,.6) {$\ss y$};}
\]
We recall from \S\ref{sec:intro} that to each tile is attached a weight $\wt_{\tp/\bt}(\text{tile},x-y)$.

\subsection{The Yang--Baxter equation and crossing symmetry}

Given parameters $x$, $x'$ and $y$, introduce an additional ``rightward'' diamond
\tikz[baseline=-3pt,xscale=.58,rotate=45]{\plaq{}\draw[-latex] (-.15,.15) -- (.15,-.15); \node at (.7,-.1) {$\ss x\,$}; \node at (.1,-.7) {$\ss x'$};}
with the following possible tiles:
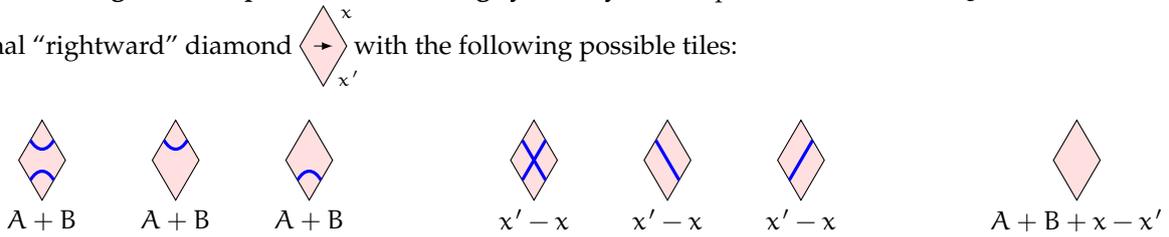
\begin{center}
\begin{tikzpicture}
\matrix[column sep=.6cm,every cell/.style={xscale=.58,rotate=-45}]{
\plaq{a} & \plaq{j} & \plaq{r} & &&
\plaq{c} & \plaq{h} & \plaq{v} & &&
\plaq{}
\\
\node{$A+B$};&\node{$A+B$};&\node{$A+B$};&&&
\node{$x'-x$};&\node{$x'-x$};&\node{$x'-x$};&&&
\node{$A+B+x-x'$};
\\
};
\end{tikzpicture}
\end{center}
We've indicated below each one its weight.
These are similar in spirit to the weights $\wt_{\tp/\bt}$.

We now have the following key result:
\begin{prop}\label{prop:ybe}
The Yang--Baxter equation holds, involving two $\tp$ diamonds and the rightward diamond:
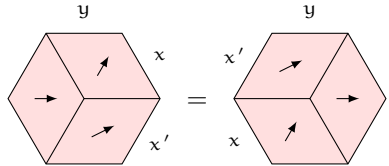
\[
\begin{tikzpicture}[baseline=-3pt]
\draw[bgplaq] (0:1) -- (60:1) -- (120:1) -- (180:1) -- (240:1) -- (300:1) -- cycle;
\draw (120:1) -- (0,0) -- (240:1) (0,0) -- (0:1);
\draw[-latex] (-.65,0) -- (-.35,0);
\draw[-latex] (60:.35) -- (60:.65);
\draw[-latex] (-60:.5) ++(-.15,-.075) -- ++(.3,.15);
\node at (30:1.15) {$\ss x$};
\node at (-30:1.15) {$\ss x'$};
\node at (90:1.15) {$\ss y$};
\end{tikzpicture}
=
\begin{tikzpicture}[baseline=-3pt]
\draw[bgplaq] (0:1) -- (60:1) -- (120:1) -- (180:1) -- (240:1) -- (300:1) -- cycle;
\draw (60:1) -- (0,0) -- (300:1) (0,0) -- (180:1);
\draw[latex-] (.65,0) -- (.35,0);
\draw[latex-] (240:.35) -- (240:.65);
\draw[-latex] (120:.5) ++(-.15,-.075) -- ++(.3,.15);
\node at (210:1.15) {$\ss x$};
\node at (150:1.15) {$\ss x'$};
\node at (90:1.15) {$\ss y$};
\end{tikzpicture}
\]
The equation should be understood as the sum of contributions of all
possible tilings of those regions, at fixed connectivity of the
boundary points.
\end{prop}

\begin{proof}
This is a direct computation. We give two examples. Suppose that we have a connectivity symbolically described by
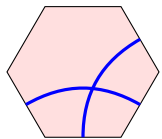
\begin{center}
\begin{tikzpicture}
\draw[bgplaq] (0:1) -- coordinate(e) (60:1) -- coordinate(f) (120:1) -- coordinate(a) (180:1) -- coordinate(b) (240:1) -- coordinate(c) (300:1) -- coordinate(d) cycle;
\draw[/linkpattern/edge,bend left] (b) to (d);
\draw[/linkpattern/edge,bend left] (c) to (e);
\end{tikzpicture}
\end{center}
The Yang--Baxter equation takes the form
\begin{align*}
\begin{tikzpicture}[baseline=-3pt]
\draw[bgplaq] (0:1) -- coordinate(e) (60:1) -- coordinate(f) (120:1) -- coordinate(a) (180:1) -- coordinate(b) (240:1) -- coordinate(c) (300:1) -- coordinate(d) cycle;
\draw (120:1) -- coordinate(h) (0,0) -- coordinate(g) (240:1) (0,0) -- coordinate(i) (0:1);
\draw[/linkpattern/edge,bend left] (b) to (g) -- (d);
\draw[/linkpattern/edge,bend left] (c) -- (i) to (e);
\end{tikzpicture}
&\ =\ 
\begin{tikzpicture}[baseline=-3pt]
\draw[bgplaq] (0:1) -- coordinate(e) (60:1) -- coordinate(f) (120:1) -- coordinate(a) (180:1) -- coordinate(b) (240:1) -- coordinate(c) (300:1) -- coordinate(d) cycle;
\draw (60:1) -- coordinate(h) (0,0) -- coordinate(i) (300:1) (0,0) -- coordinate(g) (180:1);
\draw[/linkpattern/edge,bend right] (b) to (g) [bend left] to (h) -- (d);
\draw[/linkpattern/edge,bend left] (c) to (i) -- (e);
\end{tikzpicture}
+
\begin{tikzpicture}[baseline=-3pt]
\draw[bgplaq] (0:1) -- coordinate(e) (60:1) -- coordinate(f) (120:1) -- coordinate(a) (180:1) -- coordinate(b) (240:1) -- coordinate(c) (300:1) -- coordinate(d) cycle;
\draw (60:1) -- coordinate(h) (0,0) -- coordinate(i) (300:1) (0,0) -- coordinate(g) (180:1);
\draw[/linkpattern/edge,bend left] (b) -- (i) to (d);
\draw[/linkpattern/edge,bend left] (c) -- (g) to (h) [bend right] to (e);
\end{tikzpicture}
\\
(A+B)^2(A+x'-y)&\ =\ (A+B)^2(x'-x)+(A+B)^2(A+x-y)
\end{align*}
Similarly, if the connectivity is
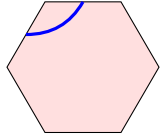
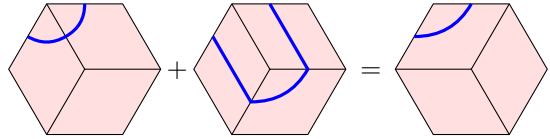
\begin{center}
\begin{tikzpicture}
\draw[bgplaq] (0:1) -- coordinate(e) (60:1) -- coordinate(f) (120:1) -- coordinate(a) (180:1) -- coordinate(b) (240:1) -- coordinate(c) (300:1) -- coordinate(d) cycle;
\draw[/linkpattern/edge,bend left] (f) to (a);
\end{tikzpicture}
\end{center}
then
\begin{align*}
\begin{tikzpicture}[baseline=-3pt]
\draw[bgplaq] (0:1) -- coordinate(e) (60:1) -- coordinate(f) (120:1) -- coordinate(a) (180:1) -- coordinate(b) (240:1) -- coordinate(c) (300:1) -- coordinate(d) cycle;
\draw (120:1) -- coordinate(h) (0,0) -- coordinate(g) (240:1) (0,0) -- coordinate(i) (0:1);
\draw[/linkpattern/edge,bend left] (f) to (h) to (a);
\end{tikzpicture}
+
\begin{tikzpicture}[baseline=-3pt]
\draw[bgplaq] (0:1) -- coordinate(e) (60:1) -- coordinate(f) (120:1) -- coordinate(a) (180:1) -- coordinate(b) (240:1) -- coordinate(c) (300:1) -- coordinate(d) cycle;
\draw (120:1) -- coordinate(h) (0,0) -- coordinate(g) (240:1) (0,0) -- coordinate(i) (0:1);
\draw[/linkpattern/edge,bend left] (f) -- (i) to (g) -- (a);
\end{tikzpicture}
&\ =\ \begin{tikzpicture}[baseline=-3pt]
\draw[bgplaq] (0:1) -- coordinate(e) (60:1) -- coordinate(f) (120:1) -- coordinate(a) (180:1) -- coordinate(b) (240:1) -- coordinate(c) (300:1) -- coordinate(d) cycle;
\draw (60:1) -- coordinate(h) (0,0) -- coordinate(i) (300:1) (0,0) -- coordinate(g) (180:1);
\draw[/linkpattern/edge,bend left] (f) to (a);
\end{tikzpicture}
\\
  (A+B)^2(B-x'+y) + (A+B)(x'-x)(A+x-y) &\ =\ (A+B)(A+B+x-x')(B-x+y)
                                         \qedhere
\end{align*} 
\end{proof}

Next we introduce another ``upward'' diamond
\tikz[baseline=-3pt,xscale=.58,rotate=45]{\plaq{}\draw[-latex] (-.1,-.1) -- (.1,.1); \node at (.7,-.1) {$\ss x\,$}; \node at (.1,-.7) {$\ss x'$};}
with the following possible fillings and weights:
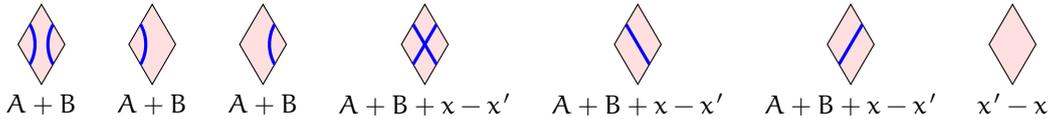
\begin{center}
\begin{tikzpicture}
\matrix[column sep=.3cm,every cell/.style={xscale=.58,rotate=-45}]{
\plaq{b} & \plaq{k} & \plaq{i} & \plaq{c} & \plaq{h} & \plaq{v} & \plaq{}
\\
\node{$A+B$};&\node{$A+B$};&\node{$A+B$};&\node{$A+B+x-x'$};&\node{$A+B+x-x'$};&\node{$A+B+x-x'$};&\node{$x'-x$};
\\
};
\end{tikzpicture}
\end{center}
We have another version of the Yang--Baxter equation,
this time mixing a $\tp$ row with a $\bt$ row:
\begin{prop}\label{prop:ybe2}
The Yang--Baxter equation holds, involving a $\tp$ diamond, a $\bt$ diamond and the upward diamond:
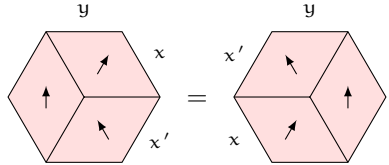
\[
\begin{tikzpicture}[baseline=-3pt]
\draw[bgplaq] (0:1) -- (60:1) -- (120:1) -- (180:1) -- (240:1) -- (300:1) -- cycle;
\draw (120:1) -- (0,0) -- (240:1) (0,0) -- (0:1);
\draw[-latex] (-.5,-.15) -- (-.5,.15);
\draw[-latex] (60:.35) -- (60:.65);
\draw[latex-] (-60:.35) -- (-60:.65);
\node at (30:1.15) {$\ss x$};
\node at (-30:1.15) {$\ss x'$};
\node at (90:1.15) {$\ss y$};
\end{tikzpicture}
=
\begin{tikzpicture}[baseline=-3pt]
\draw[bgplaq] (0:1) -- (60:1) -- (120:1) -- (180:1) -- (240:1) -- (300:1) -- cycle;
\draw (60:1) -- (0,0) -- (300:1) (0,0) -- (180:1);
\draw[-latex] (.5,-.15) -- (.5,.15);
\draw[latex-] (240:.35) -- (240:.65);
\draw[-latex] (120:.35) -- (120:.65);
\node at (210:1.15) {$\ss x$};
\node at (150:1.15) {$\ss x'$};
\node at (90:1.15) {$\ss y$};
\end{tikzpicture}
\]
where the equation should be understood as the sum of contributions of all possible tilings of those regions,
at fixed connectivity of the boundary points.
\end{prop}
\begin{proof}
The proof is similar to that of proposition~\ref{prop:ybe}; in fact, the two equations can be obtained from each
other by 60 degree rotation and appropriate substitution of parameters.
\end{proof}

One could define similarly leftward and downward diamonds, but we
shall not need them in what follows.

Finally, we need one more property:

\begin{lem}(Crossing symmetry)\label{lem:fluxr}
Consider rows which contain a single pipe.
There is a weight-preserving bijection between rows of type $\bt$ and of type $\tp$ with this property, assuming they
have the same parameter $x_i$ attached to them.
\end{lem}

\begin{proof}
The bijection is very simple to describe. Every vertical edge between two tiles can be either empty or occupied (by
the one pipe); switch its state, and then redraw the tiles accordingly. Here is a general example:
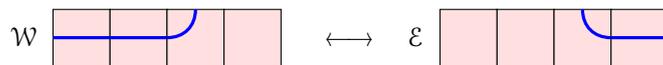
\[
\tp\ \tikz[baseline=-3pt]{\node[loop] {\plaq{h}&\plaq{h}&\plaq{j}&\plaq{}\\};}
\quad\longleftrightarrow\quad
\bt\ \tikz[baseline=-3pt]{\node[loop] {\plaq{}&\plaq{}&\plaq{i}&\plaq{h}\\};}
\]
It is easy to check that this is weight-preserving.
\end{proof}

This correspondence also appears as part of a larger bijection
in the very recent \cite{weigandt2025changing}, but we cannot see any
connection.

\subsection{The generic pipe dream polynomials}
We are now ready to analyze the generic pipe dream polynomials.
We define $G_\pi^\beta$ as in the introduction, except we make the
(soon to be proved nonexistent) dependence on $\beta$ explicit.
Pictorially, e.g., if $\beta=(\tp,\bt,\tp)$ and $\pi=435$,
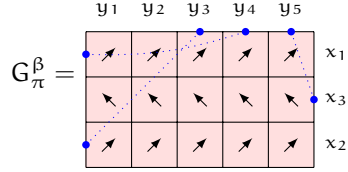
\[
G_\pi^\beta=\begin{tikzpicture}[scale=.6,baseline=(current  bounding  box.center)]
\fill[\linkpatternboxcolor] (0,0) rectangle (5,3);
\draw (0,0) grid (5,3);
\foreach\x in {1,...,5} \node at (\x-.5,3.5) {$\ss y_\x$};
\foreach\x in {1,...,5} \draw[-latex] (\x-.65,.35) -- ++(.3,.3);
\foreach\x in {1,...,5} \draw[-latex] (\x-.35,1.35) -- ++(-.3,.3);
\foreach\x in {1,...,5} \draw[-latex] (\x-.65,2.35) -- ++(.3,.3);
\begin{scope}[dotted,\linkpatternedgecolor,every node/.style={circle,inner sep=1pt,fill=\linkpatternedgecolor}]
\draw (0,.5) node {} -- (2.5,3) node {};
\draw (5,1.5) node {} -- (4.5,3) node {};
\draw[bend right=10] (0,2.5) node {} to (3.5,3) node {};
\end{scope}
\node at (5.5,0.5) {$\ss x_2$};
\node at (5.5,1.5) {$\ss x_3$};
\node at (5.5,2.5) {$\ss x_1$};
\end{tikzpicture}
\]
where it is understood that such pictures represent the ``state sum''
of weights of all fillings compatible with the boundary condition and
the connectivity $\pi$.

\begin{lem}\label{lem:beta1}
$G_\pi^\beta$ does not depend on the last letter $\beta_m$ of $\beta$.
\end{lem}
\begin{proof}
This is a direct consequence of lemma~\ref{lem:fluxr} applied to the last row of the grid, noting that the bijection
does not affect the numbering of the pipes (in particular, the parameter attached to the last row
stays the same) and the connectivity.
\end{proof}

\begin{lem}\label{lem:beta2}
One has $G_\pi^{\beta',\tp,\bt,\beta''}=G_\pi^{\beta',\bt,\tp,\beta''}$ where $\beta'$ and $\beta''$ are arbitrary subsequences.
\end{lem}
In other words, one can freely switch $\bt$s and $\tp$s.
\begin{proof}
We consider the effect of inserting one upward diamond on the left of the grid at the level of the two rows
of interest:
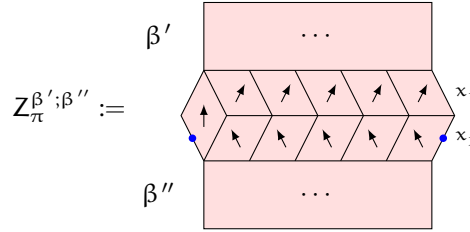
\[
Z^{\beta';\beta''}_\pi:=\begin{tikzpicture}[scale=.6,baseline=(current  bounding  box.center)]
\draw[fill=\linkpatternboxcolor] (0,0) -- (0,1.5) -- (-.5,2.5) -- (0,3.5) -- (0,5) -- (5,5) -- (5,3.5) -- (5.5,2.5) -- (5,1.5) -- (5,0) -- cycle;
\draw (0,1.5) -- (5,1.5);
\draw (.5,2.5) -- (5.5,2.5);
\draw (0,3.5) -- (5,3.5);
\foreach\x in {1,...,5} \draw (\x-1,1.5) -- (\x-.5,2.5) -- (\x-1,3.5);
\foreach\x in {1,...,5} \draw[-latex] (\x-.3,2.8) -- ++(.2,.4);
\foreach\x in {1,...,5} \draw[-latex] (\x-.2,1.8) -- ++(-.2,.4);
\draw[-latex] (0,2.25) -- ++(0,.5);
\node at (-1,.75) {$\beta''$}; \node at (2.5,.75) {$\cdots$};
\node at (-1,4.25) {$\beta'$}; \node at (2.5,4.25) {$\cdots$};
\begin{scope}[every node/.style={circle,inner sep=1pt,fill=\linkpatternedgecolor}]
\node at (-.25,2) {};
\node at (5.25,2) {};
\end{scope}
\node at (5.75,3) {$\ss x_i$};
\node at (5.75,2) {$\ss x_j$};
\end{tikzpicture}
\]
where as usual the filling must produce the connectivity $\pi$, and $a$ and $b$ are indices to be determined.

With the chosen boundary conditions, by consulting the list of allowed tiles for upward diamonds,
we note that the only possibility for the inserted diamond is \tikz[xscale=.58,rotate=-45,baseline=-3pt]{\plaq{v}}.
Removing this diamond
then reproduces the pictorial representation of $G_\pi^{\beta',\tp,\bt,\beta''}$.
We therefore declare $i$ and $j$ to be the labels of the rows $\tp$ and $\bt$ respectively ($i<j$), so that
\[
Z^{\beta';\beta''}_\pi = (A+B+x_i-x_j) G_\pi^{\beta',\tp,\bt,\beta''} 
\]
On the other hand, one can apply proposition~\ref{prop:ybe2} to $Z^{\beta';\beta''}_\pi$ repeatedly, resulting in
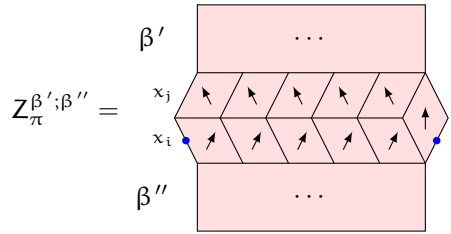
\[
Z^{\beta';\beta''}_\pi=\begin{tikzpicture}[scale=.6,baseline=(current  bounding  box.center)]
\draw[fill=\linkpatternboxcolor] (0,0) -- (0,1.5) -- (-.5,2.5) -- (0,3.5) -- (0,5) -- (5,5) -- (5,3.5) -- (5.5,2.5) -- (5,1.5) -- (5,0) -- cycle;
\draw (0,1.5) -- (5,1.5);
\draw (-.5,2.5) -- (4.5,2.5);
\draw (0,3.5) -- (5,3.5);
\foreach\x in {1,...,5} \draw (\x,1.5) -- (\x-.5,2.5) -- (\x,3.5);
\foreach\x in {1,...,5} \draw[-latex] (\x-.8,1.8) -- ++(.2,.4);
\foreach\x in {1,...,5} \draw[-latex] (\x-.7,2.8) -- ++(-.2,.4);
\draw[-latex] (5,2.25) -- ++(0,.5);
\node at (-1,.75) {$\beta''$}; \node at (2.5,.75) {$\cdots$};
\node at (-1,4.25) {$\beta'$}; \node at (2.5,4.25) {$\cdots$};
\begin{scope}[every node/.style={circle,inner sep=1pt,fill=\linkpatternedgecolor}]
\node at (-.25,2) {};
\node at (5.25,2) {};
\end{scope}
\node at (-0.75,3) {$\ss x_j$};
\node at (-0.75,2) {$\ss x_i$};
\end{tikzpicture}
\]
This time the only allowed tile for the inserted diamond is \tikz[xscale=.58,rotate=-45,baseline=-3pt]{\plaq{h}}.
Once we remove it, we recognize the pictorial representation of $G_\pi^{\beta',\bt,\tp,\beta''}$: in particular the parameters
$x_i$ and $x_j$ have been switched, as should be -- the numbering of the pipes has been modified by the change of binary string.
This results in
\[
Z^{\beta';\beta''}_\pi = (A+B+x_i-x_j) G_\pi^{\beta',\bt,\tp,\beta''} 
\]
\end{proof}

Putting together lemmas \ref{lem:beta1} and \ref{lem:beta2} immediately leads to our first theorem:

\begin{thm}\label{thm:beta}
  $G_\pi := G_\pi^\beta$ is independent of $\beta$.
\end{thm}

A similar proof appears in \cite[theorem 3]{KU-hybrid}, which concerns
Schubert polynomials rather than generic pipe dream polynomials.
Instead of the Yang--Baxter-based argument here, that proof uses a
bijection. However, it only applies to single Schubert polynomials
$S_\pi(0,\ul y)$. While a bijective proof of the double Schubert
polynomial may be possible, it would require decoration as in
\cite{DaojiBijective}, which the Yang--Baxter-based argument does not.

The implication of proposition \eqref{prop:ybe} (the other Yang--Baxter
equation, where both rows are type $\tp$) for $G_\pi$ is more subtle,
and is the subject of the next theorem.
Recall that the inversion number of $\pi$ is
\[
  \ell(\pi) := \# \left\{i<j \mid \pi(j)>\pi(i) \right\}
\]
Write $r_i$ for the elementary transposition $(i\ i+1)$,
and make it act on $\ZZ[x_1,\ldots,x_m]$ by permutation of variables.

\begin{thm}\label{thm:ind}
  $G_\pi$ is uniquely determined by the following inductive definition:
  \begin{itemize}
  \item If $\pi$ is decreasing (that is, $\ell(\pi)=m(m-1)/2$ is maximal), then
    \begin{equation}\label{eq:base}
      G_\pi=\prod_{i=1}^m \left((A+B)\prod_{j=1}^{\pi(i)-1} (A+x_i-y_j)\prod_{j=\pi(i)+1}^n (B-x_i+y_j)\right)
    \end{equation}
  \item Given $\pi\in\SS_{m,n}$, define $\pi'=\pi r_i$ by switching
    $\pi(i)$ and $\pi(i+1)$. Then
    \begin{equation}\label{eq:ind}
      G_{\pi} = \frac{1}{x_i-x_{i+1}}((A+B) G_{\pi'}-(A+B+x_i-x_{i+1})r_i G_{\pi'})
    \end{equation}
  \end{itemize}
\end{thm}
\begin{proof}
  In this proof, we shall only use ``non-hybrid'' generic pipe dreams,
  that is, the binary string $\beta=(\tp,\ldots,\tp)$. In particular
  the numbering of pipes $\varphi(i)$ coincides with the row number $i$.
  (Of course theorem \ref{thm:beta} implies that our recurrence
  (\ref{eq:ind}) holds for arbitrary hybridizations.)

  The first part follows immediately by inspection of the (unique) generic pipe dream corresponding
  to a decreasing $\pi$,
  e.g., for $\pi=431$,
  
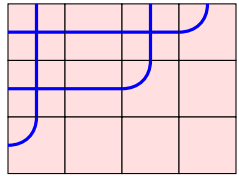
\begin{center}
    \begin{tikzpicture}
      \node[loop]{\plaq{c}&\plaq{h}&\plaq{c}&\plaq{j}\\\plaq{c}&\plaq{h}&\plaq{j}&\plaq{}\\\plaq{j}&\plaq{}&\plaq{}&\plaq{}\\};
    \end{tikzpicture}
  \end{center}
  Its contribution reproduces \eqref{eq:base}.
  
  For the second part, we consider the following picture
  \[
    Z_{\pi'}:=
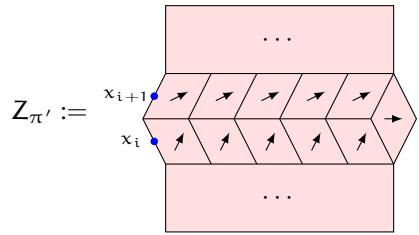
\begin{tikzpicture}[scale=.6,baseline=(current  bounding  box.center)]
\draw[fill=\linkpatternboxcolor] (0,0) -- (0,1.5) -- (-.5,2.5) -- (0,3.5) -- (0,5) -- (5,5) -- (5,3.5) -- (5.5,2.5) -- (5,1.5) -- (5,0) -- cycle;
\draw (0,1.5) -- (5,1.5);
\draw (-.5,2.5) -- (4.5,2.5);
\draw (0,3.5) -- (5,3.5);
\foreach\x in {1,...,5} \draw (\x,1.5) -- (\x-.5,2.5) -- (\x,3.5);
\foreach\x in {1,...,5} \draw[-latex] (\x-.8,1.8) -- ++(.2,.4);
\foreach\x in {1,...,5} \draw[-latex] (\x-.9,2.9) -- ++(.4,.2);
\draw[-latex] (4.8,2.5) -- ++(0.4,0);
\node at (2.5,.75) {$\cdots$};
\node at (2.5,4.25) {$\cdots$};
\begin{scope}[every node/.style={circle,inner sep=1pt,fill=\linkpatternedgecolor}]
\node at (-.25,2) {};
\node at (-.25,3) {};
\end{scope}
\node at (-0.85,3) {$\ss x_{i+1}$};
\node at (-0.75,2) {$\ss x_i$};
\end{tikzpicture}
\]
where we impose connectivity $\pi'$, and the $i^{\rm th}$ and
$(i+1)^{\rm st}$ rows that are displayed have been pre\"emptively
labeled with switched parameters $x_{i+1}$ and $x_i$.

Once again, there is a unique tile for the inserted diamond namely
the blank tile \tikz[xscale=.58,rotate=-45,baseline=-3pt]{\plaq{}},
and removing it results in the pictorial representation of $G_{\pi'}$
itself. Therefore $Z_{\pi'}$ is the weight of the blank tile times
$G_\pi$, except the latter has $x_i$ and $x_{i+1}$ exchanged:
\[
Z_{\pi'}= (A+B+x_i-x_{i+1})r_i G_{\pi'}
\]

Now apply repeatedly proposition~\ref{prop:ybe}: we find
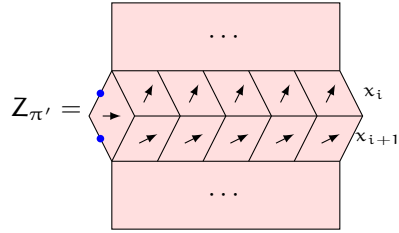
\[
  Z_{\pi'} =
  \begin{tikzpicture}[scale=.6,baseline=(current  bounding  box.center)]
\draw[fill=\linkpatternboxcolor] (0,0) -- (0,1.5) -- (-.5,2.5) -- (0,3.5) -- (0,5) -- (5,5) -- (5,3.5) -- (5.5,2.5) -- (5,1.5) -- (5,0) -- cycle;
\draw (0,1.5) -- (5,1.5);
\draw (.5,2.5) -- (5.5,2.5);
\draw (0,3.5) -- (5,3.5);
\foreach\x in {1,...,5} \draw (\x-1,1.5) -- (\x-.5,2.5) -- (\x-1,3.5);
\foreach\x in {1,...,5} \draw[-latex] (\x-.3,2.8) -- ++(.2,.4);
\foreach\x in {1,...,5} \draw[-latex] (\x-.4,1.9) -- ++(.4,.2);
\draw[-latex] (-.2,2.5) -- ++(0.4,0);
\node at (2.5,.75) {$\cdots$};
\node at (2.5,4.25) {$\cdots$};
\begin{scope}[every node/.style={circle,inner sep=1pt,fill=\linkpatternedgecolor}]
\node at (-.25,2) {};
\node at (-.25,3) {};
\end{scope}
\node at (5.75,3) {$\ss x_{i}$};
\node at (5.85,2) {$\ss x_{i+1}$};
\end{tikzpicture}
\]
At the end, we find a situation where the inserted diamond has two
options: either \tikz[xscale=.58,rotate=-45,baseline=-3pt]{\plaq{c}}
or \tikz[xscale=.58,rotate=-45,baseline=-3pt]{\plaq{a}}.
In the first case, removing it changes the connectivity from $\pi'$ to
$\pi=\pi' r_i$, whereas the second case leaves the connectivity
unchanged. All in all, one finds
\[
Z_{\pi'}=(x_{i+1}-x_{i}) G_\pi + (A+B) G_{\pi'}
\]
Combining those two equalities results in \eqref{eq:ind}.

Finally, by induction on $m(m-1)/2-\ell(\pi)$, \eqref{eq:base} and \eqref{eq:ind} define $G_\pi$ uniquely:
\eqref{eq:base} is the base case,
and if $\ell(\pi)<m(m-1)/2$, there exists an $i$ such that $\pi(i)<\pi(i+1)$, so that \eqref{eq:ind}
expresses $G_\pi$ in terms of $G_{\pi'}$ with 
$\ell(\pi')=\ell(\pi)+1$.
\end{proof}

\subsection{The mirror image map}\label{ssec:mirror}
We describe a mapping between ``opposite'' hybridizations:
\begin{lem}\label{lem:mirror}
Let $\beta\in \{\tp,\bt\}^m$, and $\bar\beta$ be the opposite hybridization defined by $\bar\beta_i\ne \beta_i$ for all $i=1,\ldots,m$.
Then the operation consisting in taking the left-right mirror image
is a bijection between hybrid generic pipe dreams of type $\beta$ with connectivity $\pi$
and
hybrid generic pipe dreams of type $\bar\beta$ with connectivity $\gamma_n\pi\gamma_m$.
\end{lem}
\begin{proof}
  First, one checks that the allowed tiles on rows of types $\tp$ and
  $\bt$ are mirror images of each other. Secondly, the sides on which pipes enter (West vs East) are exchanged by
  the mirror image map, hence $\beta\leftrightarrow\bar\beta$. Finally, not only are ending locations of pipes on the North flipped
  (hence the $\gamma_n$), but also their numbering is reversed (hence the $\gamma_m$).
\end{proof}

As an immediate corollary, by comparing the contributions of a pipe dream and of its mirror image to $G_\pi$ and $G_{\gamma_n\pi\gamma_m}$
respectively, one finds:
\begin{cor}The following equality holds:
\[
G_\pi (A,B,x_1,\ldots,x_m,y_1,\ldots,y_n)
=
G_{\gamma_n\pi\gamma_m}(B,A,-x_m,\ldots,-x_1,-y_n,\ldots,-y_1)
\]
\end{cor}

This is to be compared with \S\ref{ssec:invo}.

\section{Pipe dreams and equivariant cohomology}

Our reference for equivariant cohomology is \cite{AndersonFulton}.

\subsection{A divided difference operator recurrence}

We denote by $[X]$ the torus-equivariant cohomology class of a subvariety $X$
of a vector space $V$ with a torus action, and identify it with a polynomial
via the isomorphism
$H^*_T(V)\cong H^*_T(pt)\cong \text{Sym}(T^*)$.  We recall a basic
lemma from \cite{BBM} about the geometry of divided difference
operators, as used already in \cite[\S4.5]{artic32},
\cite[\S4]{artic33}.  For notational convenience, we state it in less
than maximum generality.

\begin{lem}\label{lem:divdiff}
  Let $T'\times GL_2$ act on a vector space $V$, and $X \subseteq V$
  a subvariety invariant under the subgroup $T' \times B^-_2$
  (where $B^-_2\leq GL_2$ is a Borel subgroup, and $T_2\leq B^-_2$ a
  Cartan subgroup). Let $T = T' \times T_2$, so $X$ is $T$-invariant.
  Then
  \begin{itemize}
  \item $GL_2\cdot X \subseteq V$ is closed and irreducible,
  \item Under the isomorphism $H^*_{T}(V) \iso H^*_{T'}(pt)\tensor H^*_{B^-_2}(pt) \iso
    H^*_T(pt)[p_1,p_2]$, the natural action of $r \in W(GL_2) \iso \ZZ/2$
    is by $p_1 \leftrightarrow p_2$. 
  \end{itemize}
  For $x\in X$ a general point,
  let $d := \#\{ B^-_2\backslash B^-_2\, g\in GL_2,\ g\colon g\cdot x \in X\}$.
  There are two cases:
  \begin{enumerate}
  \item If $d = \infty$, then $[X] = [X]|_{p_1\leftrightarrow p_2}$.
    In particular this occurs if $X$ is $(T \times GL_2)$-invariant.
  \item If $d < \infty$ (as happens if $X$ is not $GL_2$-invariant),
    then $[GL_2\cdot X] 
    = \frac{1}{(p_2-p_1)d}\left( [X] - r\cdot [X] \right)$.
    Also, $GL_2\cdot X$ contains $X$ in codimension $1$.
  \end{enumerate}
\end{lem}

In most cases of interest $d$ turns out to be $1$. We now apply this lemma to our setting.

\begin{thm}\label{thm:ind2}
  The equivariant cohomology classes $[E_\pi]$ satisfy the following:
  \begin{itemize}
  \item If $\pi$ is decreasing,
    \begin{equation}\label{eq:base2}
      [E_\pi] = \prod_{i=1}^m \left(\prod_{j=1}^{\pi(i)-1} (A+x_i-y_j)\prod_{j=\pi(i)+1}^n (B-x_i+y_j)\right)
    \end{equation}
  \item Given $\pi\in\SS_{m,n}$ and $\pi'=\pi r_i$,
    \begin{equation}\label{eq:ind2}
      [E_{\pi'}] = \frac{1}{x_i-x_{i+1}} \bigg(
        (A+B)\, [E_{\pi}] - (A+B+x_i-x_{i+1})\, [E_{\pi}]|_{x_i\leftrightarrow x_{i+1}}
      \bigg)
    \end{equation}
  \end{itemize}
\end{thm}
\begin{proof}
  For the first part, if $\pi$ is decreasing, then among the rank equations \textit{(iii)} and \textit{(iv)}
  of proposition~\ref{prop:eqs}, we
  find the linear equations
  \begin{align*}
    X_{ij}&=0\qquad j<\pi(i)
    \\
    Y_{ji}&=0\qquad j>\pi(i)
  \end{align*}
  There are $m(n-1)$ such equations, which is the codimension of
  $E_\pi$; therefore $E_\pi$ is the linear subspace given by these
  equations, and its $H_T^*$ class is the product of their weights,
  which is precisely \eqref{eq:base2}.
  
  As to the second part, we first rewrite \eqref{eq:ind2} in the more
  suggestive form
  \[
    \qquad\qquad
    - (A+B+x_i-x_{i+1}) \partial_i [E_{\pi}] = [E_{\pi}] +[E_{\pi'}],
    \qquad\quad\text{where}\quad
    \partial_i := \frac{1}{x_i-x_{i+1}}(1-r_i)
  \]
  is the \textbf{divided difference operator}. This equation is now ripe for
  lemma \ref{lem:divdiff}.
  Take $X = E_{\pi}$, let
  $V = \Mat(m,n,\CC)\times \Mat(n,m,\CC)$, the $GL_2$ action has
  $g\cdot(X,Y) := \left(
    \begin{bmatrix}      I_{i-1} \\ & \!g\! \\ && I_{m-i-1}     \end{bmatrix} X,
    Y\begin{bmatrix}     I_{i-1} \\ & \!g^{-1}\! \\ && I_{m-i-1}     \end{bmatrix} 
  \right)$, and $T' \leq B_- \times B_+$ consists of pairs $(D_1,D_2)$
  of invertible diagonal matrices where $(D_1)_{ii} = (D_1)_{i+1,i+1} = 1$.
  Hereafter identify $GL_2$ with its image inside $GL_m$. 
  
  Let $Z := (GL_2)\cdot E_{\pi}$, and call a typical element $(X,Y)$.
  Using proposition \ref{prop:irr} one can see that the functions
  $(XY)_{ii},(XY)_{i+1,i+1}$ on $E_\pi$ are independent, and hence,
  the function $(XY)_{i,i+1}$ on $Z$ is not always zero.
  However, $Z$ continues to satisfy all the other vanishing conditions
  defining the lower-upper scheme $E$.
  By failing to satisfy this one though,
  $Z$ properly contains $E_{\pi}$, hence by
  lemma \ref{lem:divdiff} has $\dim Z = \dim E_{\pi} + 1$.

  Let $Z' = Z \cap \{ (X,Y) \colon (XY)_{i,i+1} = 0 \}$,
  so $Z' \subseteq E$. Since $Z$ is irreducible, the Cartier divisor $Z'$
  is pure of codimension $1$ in $Z$. Since $E$ is equidimensional
  (theorem \ref{thm:Eequidim} in the appendix) of
  dimension $\dim E_{\pi} = \dim Z - 1 = \dim Z'$, we learn $Z'$
  must be a (potentially schemy) union of components of $E$.
  Since $Z' \subseteq E$ scheme-theoretically and $E$ is generically reduced
  (lemma \ref{lem:genred}), we learn that $Z'$ must likewise be
  generically reduced.

  It remains to show that $Z' = E_\pi \cup E_{\pi'}$, up to the possible
  appearance of embedded components (which we conjecture are not there),
  and to check that $d=1$ in the notation of lemma \ref{lem:divdiff}.
  Enough of the equations of proposition \ref{prop:eqs}(ii) hold on $Z$ to
  exclude (using the end of that proposition)
  any component other than $E_{\pi}$, $E_{\pi'}$.
  To show that $E_\rho \subseteq Z$, one need only find a point in
  $Z \setminus \Union_{\sigma \neq \rho} E_\sigma$. So pick a diagonal
  matrix $s \in GL(m)$ with distinct diagonal entries, and observe that
  $(\pi, \pi^T s) \in Z \setminus \Union_{\sigma \neq \pi} E_\sigma$, whereas
  $(r_i \pi, \pi^T s r_i) \in Z \setminus \Union_{\sigma \neq \pi'} E_\sigma$.

  Finally, to show that $d=1$ (the degree from lemma \ref{lem:divdiff}),
  it suffices to find one $x \in E_{\pi}$ such
  that $\{ B_2^-\backslash B_2^-\, g\colon g\in GL_2,\ g\cdot x \in X\}$ is a
  reduced point. Take $x$ to be the pair $(\pi,\pi^Ts)$ from above,
  with products $\pi \pi^T s = s$, $\pi^T s \pi$. Applying $g$ changes
  these products to $g s g^{-1}$, $\pi^T s g^{-1} g \pi = \pi^T s \pi$.
  For $(g \pi, \pi^T s g^{-1})$ to be in $E_\pi$,
  we need $g s g^{-1}$ to still be lower triangular, and with the same
  diagonal entries.

  This is then a computer algebra calculation.
  Let $S =  \begin{bmatrix}     a & 0 \\ 0 & e  \end{bmatrix}$ be the
  $2\times 2$ submatrix of $s$ with rows and columns $(i,i+1)$, and
  $g = \begin{bmatrix}     p & q \\ r & s  \end{bmatrix}$.
  Then the ideal
  $I = \langle (gSg^{-1})_{12} = 0, gSg^{-1})_{11} = a, gSg^{-1})_{22} = d \rangle$
  is computed to have four components: the one we want (with $q=0$),
  one with $a=e=0$, one with $a=e$, and another with $a=-e$.
  So for $a,e$ general enough to avoid those latter three
  possibilities, we must have $q=0$, i.e. $g\in B_2^-$. Hence $d=1$.  

  We have now laid the groundwork to conclude
  \begin{eqnarray*}
    [E_\pi] + [E_{\pi'}]
    &=& [Z'] \\
    &=& [Z]\, [\{ (X,Y) \colon (XY)_{i,i+1} = 0 \}] \\
    &=& (-\partial_i [E_{\pi'}])\, [\{ (X,Y) \colon (XY)_{i,i+1} = 0 \}]
        \qquad\text{using lemma \ref{lem:divdiff}} \\
    &=& (-\partial_i [E_{\pi'}])\, (A+B+x_i-x_{i+1})
  \end{eqnarray*}
  as we had hoped.
\end{proof}

Comparing theorems~\ref{thm:ind} and \ref{thm:ind2} leads to the obvious corollary
\begin{cor}[part \textit{(2)} of theorem~\ref{thm:main}] \label{cor:main}
  $     G_\pi  = (A+B)^m \, [E_\pi]    $.
\end{cor}

\subsection{Positivity of the weights}\label{ssec:pos}
We pause here to mention an important positivity property of the weights
introduced in \S\ref{sec:intro}, which will be used twice in what follows.
In particular we spell out the sense in which our pipe dream
formula for $[E_\pi]$ is ``positive''.

Recall that a \defn{rig} is a ``ri{\bf n}g without {\bf n}egatives'',
closed under $+$ and $\cdot$ but not necessarily $-$.
Let $M$ be the subrig of $\ZZ[x_1,\ldots,x_m,y_1,\ldots,y_n,A,B]$ 
consisting of sums of products of weights $\wt_T(\bullet)$ from \eqref{eq:wt}
(or equivalently of weights $\wt_{\tp/\bt}(\bullet)$; see the remark right after
that equation).
Define a preorder relation on $\ZZ[x_1,\ldots,x_m,y_1,\ldots,y_n,A,B]$
by $P\le Q$ iff $Q-P\in M$. This $M$ satisfies $M\cap -M=0$
(proof: consider the specialization $A=B=1$, $x_i = y_j = 0$),
which is equivalent to saying that $\le$ is in fact an order relation.

In other words, the multigrading of the co\"ordinate ring of
$\Mat(m,n,\CC)\times \Mat(n,m,\CC)$ defined by these weights 
is \textbf{positive} in the sense of \cite[theorem 8.6]{MS-book},
so we can write
\[
  X\subseteq Y\text{ and }\dim X=\dim Y\qquad\Rightarrow\qquad [X] \le [Y]
\]
with equality iff $X$ and $Y$ differ only by lower-dimensional
(possibly embedded) components.

\subsection{Nongeneric hybrid pipe dreams}\label{ssec:nongeneric}
The definition of a nongeneric hybrid pipe dream (as introduced in
\cite{KU-hybrid}) is very close to that of a generic hybrid pipe dream.
A (nongeneric) hybrid pipe dream of type $\beta$ is a filling of the
$m\times n$ grid with the following tiles:
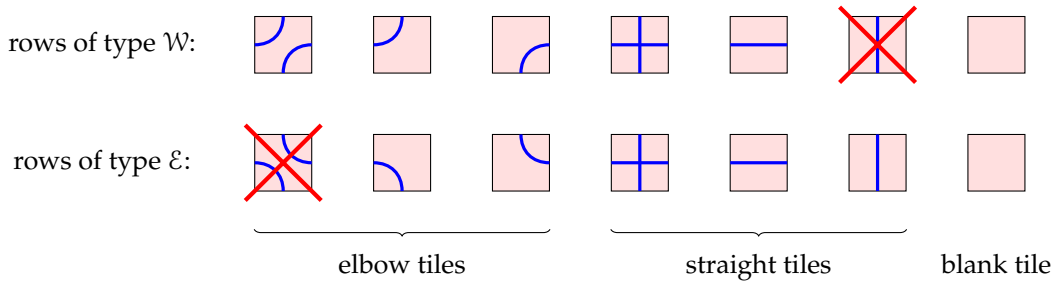
\tikzset{loopmod/.code={\def\plaqname{loop-\the\pgfmatrixcurrentrow-\the\pgfmatrixcurrentcolumn}},loop/.append
  style={matrix}}
\begin{center}
  \begin{tikzpicture}
\node[loop,column sep=.8cm,row sep=.8cm] {
\plaq{a} & \plaq{j} & \plaq{r} & \plaq{c} & \plaq{h} & \plaq{v} \draw[red,ultra thick] (-.5,-.5) -- (.5,.5) (-.5,.5) -- (.5,-.5); & \plaq{}
\\
\plaq{b} \draw[red,ultra thick] (-.5,-.5) -- (.5,.5) (-.5,.5) -- (.5,-.5); & \plaq{k} & \plaq{i} & \plaq{c} & \plaq{h} & \plaq{v} & \plaq{}
\\
};
\draw[decoration={brace},decorate] ([yshift=-5mm]loop-2-3.south east) -- node[below=2mm] {elbow tiles} ([yshift=-5mm]loop-2-1.south west);
\draw[decoration={brace},decorate] ([yshift=-5mm]loop-2-6.south east) -- node[below=2mm] {straight tiles} ([yshift=-5mm]loop-2-4.south west);
\draw[decoration={brace},decorate,draw=none] ([yshift=-5mm]loop-2-7.south east) -- node[below=2mm] {blank tile} ([yshift=-5mm]loop-2-7.south west);
\path (loop-1-1.west) -- ++(-2,0) node {rows of type $\tp$:};
\path (loop-2-1.west) -- ++(-2,0) node {rows of type $\bt$:};
\end{tikzpicture}
\end{center}
in such a way that
\begin{itemize}
\item No pipe starts or terminates in the middle of the grid;
\item \textcolor{red}{No two pipes cross twice;}
\item In rows of type $\tp$, a pipe enters from the West end of that row, but none exit from the East end;
\item In rows of type $\bt$, a pipe enters from the East end of that row, but none exit from the West end;
\item There are no pipes coming in from the South side.
\end{itemize}
We have emphasized in red the difference with the generic case.

Given a hybridization $\beta$,
define $S_\pi\in \ZZ[x_1,\ldots,x_m,y_1,\ldots,y_n]$ associated to $\pi\in\SS_{m,n}$ to be the sum
over nongeneric hybrid pipe dreams of type $\beta$ with connectivity $\pi$ of the following product over tiles at location $(i,j)$:
\begin{itemize}
\item If $\beta_i=\tp$, straight tiles 
contribute $x_{\varphi(i)}-y_j$;
\item If $\beta_i=\bt$,
blank tiles contribute $x_{\varphi(i)}-y_j$;
\item All other tiles contribute $1$.
\end{itemize}

Given $\pi\in \SS_{m,n}$, let $j_1<\cdots<j_{n-m}$
be the zero columns of the corresponding partial permutation matrix,
and define the ``minimal extension'' $\hat\pi\in \SS_n$
by $\hat\pi(i)=\pi(i)$, $1\le i\le m$, and $\hat\pi(i+m)=j_i$, $i=1,\ldots,n-m$.
E.g., if $\pi=24\in\SS_{2,4}$, $\hat\pi=2413$.

Extending \cite[theorem~1]{KU-hybrid} to $m<n$, one has the following:

\begin{prop}\label{prop:schub}
$S_\pi$ is the double Schubert polynomial associated to the permutation $\hat\pi$.
\end{prop}

\begin{proof}
  On the one hand, according to \cite[theorem 1]{KU-hybrid},
  $S_{\hat\pi}$ is the double Schubert polynomial associated to the
  permutation $\hat\pi$. We now want to reduce the general case
  to this $m=n$ case.

  Extend the hybridization $\beta$ to
  $\hat\beta=(\underbrace{\bt,\ldots,\bt}_{n-m})\cup\beta$ and
  consider the nongeneric hybrid pipe dreams of type $\hat\beta$
  associated to $\hat\pi$. If $\pi=24$ and $\beta=(\tp,\bt)$, this could be
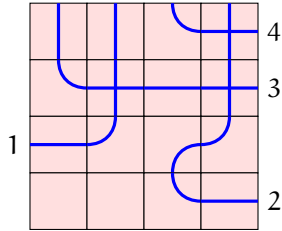
\begin{center}
  \begin{tikzpicture}\node[loop]{
\plaq{v}&\plaq{v}&\plaq{i}&\plaq{c}\\
\plaq{i} &\plaq{c}&\plaq{h}&\plaq{c}\\
\plaq{h}&\plaq{j}&\plaq{r}&\plaq{j}\\
\plaq{}&\plaq{}&\plaq{i}&\plaq{h}\\
};
\path (loop-3-1.west) node[left] {$1$};
\path (loop-4-4.east) node[right] {$2$};
\path (loop-2-4.east) node[right] {$3$};
\path (loop-1-4.east) node[right] {$4$};
\end{tikzpicture}
\end{center}
The important observation is that the first $n-m$ rows of any such pipe dream are ``frozen'', in the sense that they are always the same:
they consist of tiles \plaqctr{i} at $(i,\hat\pi(n+1-i))$, $i=1,\ldots,m$, and of straight tiles everywhere else.
This can be easily proven inductively row by row, starting from the top. Here it is crucial that the ``bump'' tile \plaqctr{b} is forbidden.

In particular, there are no blank tiles in the first $n-m$ rows, so the contribution of any such nongeneric hybrid pipe dream
to $S_{\hat\pi}$ is concentrated in the last $m$ rows, where it restricts to a nongeneric hybrid pipe dream of type $\beta$ associated
to $\pi$. This immediately implies $S_{\hat\pi}=S_\pi$.
\end{proof}

The $B$-leading form of $G_\pi$ is by definition the sum of terms of $G_\pi$ that contain a maximal power of $B$.
We can now formulate the main theorem of this section:
\begin{thm}\label{thm:nongen}
Given $\pi\in\SS_{m,n}$, the $B$-leading form of $G_\pi$ is $B^{mn-\ell(\hat\pi)}\, S_\pi(A+x_1,\ldots,A+x_m,y_1,\ldots,y_n)$.

Furthermore, only nongeneric hybrid pipe dreams contribute to the $B$-leading form of $G_\pi$,
and their contribution is of the form $B^{mn-\ell(\hat\pi)}$ times their contribution to $S_\pi$.
\end{thm}
\begin{proof}
Recall the projection $p_1:(X,Y)\mapsto X$ from $V\oplus W$ to $V$ where $V=\Mat(m,n,\CC)$ and $W=\Mat(n,m,\CC)$.
We split the $T$-action into $T'\times \CC^\times$ where $\text{Sym}(T'{}^*)=\ZZ[x_1,\ldots,x_m,y_1,\ldots,y_n,A]$, and $\CC^\times$
which acts on $V\oplus W$ with weight $0$ on $V$, weight $1$ on $W$, and corresponds to the $B$ variable.

We are in the set-up of \cite[\S3.1]{artic39}: consider
the projection $\barX_\pi:=p_1(E_\pi)$ of $E_\pi$ (which is closed and equal to $E_\pi \cap V$).
We apply \cite[proposition~3]{artic39}: the $B$-leading form of $[E_\pi]$ is $\mu \, B^d \, [\barX_\pi]$,
where $\mu\in\ZZ_{\ge1}$ is the degree of a generic fiber of $p_1$, $d=\dim W + \dim \barX_\pi-\dim E_\pi$ and $[\barX_\pi]\in \text{Sym}(T'{}^*)$.

At a generic point of $\barX_\pi$, the fiber in $E_\pi$ is the same as the fiber in $E$;
since the equations of $E$ are linear in $Y$,
the fiber is a linear subspace, hence $\mu=1$.

As is clear from the characterization of $E_\pi$ in proposition~\ref{prop:irr},
$\barX_\pi=p_1(E_\pi)$ is the \textbf{matrix Schubert variety} $\overline{B_- \pi B_+}$
\cite{Ful-flags}.
(In particular, rank equations \textit{(iii)} in proposition~\ref{prop:eqs} are nothing but the defining equations for $X\in\barX_\pi$.)

By definition, the double Schubert polynomial $S_\pi\in\ZZ[x_1,\ldots,x_m,y_1,\ldots,y_n]$ is the equivariant cohomology class
of $\barX_\pi$, where the torus action is left/right multiplication by invertible diagonal matrices.
\footnote{It is customary to extend matrix Schubert varieties with
  $m\leq n$ to square matrices by introducing irrelevant entries
  $X_{ij}$ with $i=m+1,\ldots,n$; the resulting variety is nothing but
  $\barX_{\hat\pi}$, as can be checked by noting that no new rank
  conditions of the type of proposition~\ref{prop:eqs} \textit{(iii)}
  are introduced by the extension from $\pi$ to $\hat\pi$. That is the
  geometry behind our proof of proposition~\ref{prop:schub}, where we
  effectively showed equality of the double Schubert polynomials associated to
  $\pi$ and $\hat\pi$.}
Compared to our $T'$-action, this corresponds
to setting $A=0$, so that $S_{\pi} = [\barX_\pi]|_{A=0}$; but
actually, the action of $T'$ on $V$ factors through a codimension one
subtorus, allowing us to restore the $A$-dependence by say
$x_i\mapsto A+x_i$. The dimensions are also known: $\dim W=mn$,
$\dim \barX_\pi=mn-\ell(\hat\pi)$, and $\dim E_\pi = m(n+1)$, so
$d=m(n-1)-\ell(\hat\pi)$.

The final formula for the $B$-leading form of $[E_\pi]$ is $B^{m(n-1)-\ell(\hat\pi)} S_\pi(A+x_1,\ldots,A+x_m,y_1,\ldots,y_n)$.
Using corollary~\ref{cor:main} gives us the desired $B$-leading form of $G_\pi$.

For the second part of the theorem, consider the $B$-leading form of the weight of tiles:
\begin{itemize}
\item If $\beta_i=\tp$,
blank tiles contribute $B$, straight tiles 
contribute $A+x_{\varphi(i)}-y_j$;
\item If $\beta_i=\bt$,
blank tiles contribute $A+x_{\varphi(i)}-y_j$, straight tiles 
contribute $B$;
\item Elbow tiles contribute $B$.
\end{itemize}
This matches the weights described above in the definition of $S_\pi$ up to the shift $x_i\mapsto A+x_i$ and
a power of $B$; by homogeneity of $G_\pi$ and $S_\pi$,
we can determine the overall power of $B$ to be $\deg G_\pi - \deg S_\pi = mn-\ell(\hat\pi)$.
Therefore, the $B$-leading form of the sum over nongeneric pipe dreams only, reproduces $B^{mn-\ell(\hat\pi)} S_\pi$, i.e.,
the $B$-leading form of the full sum.

Note that the weights $A+x_i-y_j$ above generate a subrig of the rig of all
weights, so we can apply positivity from \S\ref{ssec:pos} to immediately conclude that
no other pipe dream can contribute to the $B$-leading form of the sum.
\end{proof}

\junk{
  One might expect a more direct proof that $[\barX_\pi]$ arises as the
  $B$-leading form of $[E_\pi]$; indeed, this follows from $\barX_\pi$ being the
  projection of $E_\pi$, via \cite[proposition 3]{artic39}.
  }

\section{The degeneration}\label{sec:degen}
The proof of part \textit{(1)} of theorem~\ref{thm:main} will be inductive, building the pipe dreams row by row from top to bottom,
so that the induction is on $m$.
We describe the corresponding geometry now, splitting into
cases depending on the value of $\beta_1\in \{\bt,\tp\}$.

\subsection{$\tp$-degenerating the equations of $E$}\label{ssec:tdegen}
If $\beta_1=\tp$, we will degenerate the first row (resp.\ column) of $X$ (resp.\ $Y$). 
We introduce the following weights for the co\"ordinates on $\Mat(m,n,\CC)\times \Mat(n,m,\CC)$:
\[
\wt(X_{ij})=\begin{cases}
j&i=1
\\
0&\text{else}
\end{cases}
\qquad
\wt(Y_{ji})=\begin{cases}
-j&i=1
\\
0&\text{else}
\end{cases}
\]
(These should not be confused with the $T$-weights $\wt_T$ defined in \eqref{eq:wt}; they correspond to a new circle action
under which $E$ is {\em not}\/ invariant.)

We shall now study the \defn{initial ideal} of the ideal of equations of $E$ w.r.t.\ the weights above, i.e., the ideal
generated by its initial terms (terms with highest weight). Because the associated monomial ordering is only a (total) preorder (i.e.,
there can be ties),
the resulting initial ideal will in general not be monomial.
(Some authors would therefore not call this a ``Gr\"obner degeneration'',
but it should at least be called a  ``partial Gr\"obner degeneration''.)

Also define for future use the matrices $X'$ (resp.\ $Y'$) obtained from $X$ (resp.\ $Y$) by deleting the first
row (resp.\ column), as well as $E'$ which is the lower-upper scheme for $m\mapsto m-1$.

We shall use below the notation $t_i:=(XY)_{ii}$;
recall its interpretation as ``flux of the pipe numbered $i$''.

\subsubsection{The na\"\i ve limit of the equations}
Consider the equations of $E$:
\[
(XY)_{ii'}=0\qquad i<i' \qquad (YX)_{jj'}=0\qquad j>j'
\]
The $XY$ equations are unchanged if $i>1$ and can be reformulated as: $X'Y'$ lower triangular.
\rem[gray]{I don't need $i=1$, because we allow ourselves lower-dimensional junk}

Now write the $YX$ equations explicitly:
\[
Y_{j1}X_{1j'}+\sum_{i=2}^m Y_{ji}X_{ij'}=0 
\]
The weight of $Y_{j1}X_{1j'}$ is $-j+j'<0$, so the first term drops out and we are left with the condition:
$Y'X'$ upper triangular.

In short, among the equations of the flat limit of $E$, one can find
the equations of $E'$.

\subsubsection{Overlap and fluxes}\label{ssec:defflux}
We need more equations obtained by overlap of initial terms
(i.e. we perform one step of the Buchberger algorithm).
Consider the equation $(XY)_> X - X(YX)_<=0$, where
$()_>$ (resp.\ $()_<$) denotes strict lower (resp.\ upper) triangle. The $(1,j)$ entry of that equation reads
\[
\sum_{i'>1,j'} X_{1j'}Y_{j'i'}X_{i'j}
-
\sum_{i',j'>j} X_{1j'}Y_{j'i'}X_{i'j}
=0
\]
Now degenerate this equation. Paying attention to the cancelation that
occurs between the two sums, we find that in the first sum only the
terms where $j'=1$ survives, whereas in the second, only the terms
where $i'=1$ survive; leading to
\begin{equation}\label{eq:degenflux}
X_{1j}\left(\sum_{i'>1}Y_{ji'}X_{i'j} - \sum_{j'>j} X_{1j'}Y_{j'1}\right)=0
\end{equation}

This equation has an important interpretation that we describe now.
Introduce \textbf{flux} variables on the edges of a row (i.e., a $1\times n$ square grid) of type $\tp$
as follows:
\begin{itemize}
\item On the top horizontal edge on column $j$, define the flux to be $(YX)_{jj}$.
\item On the bottom horizontal edge on column $j$, define the flux to be $(Y'X')_{jj}$.
\item On the vertical edges, numbered from $0$ to $n$, define the flux to be $\sum_{j'>j} X_{1j'}Y_{j'1}$.
\end{itemize}
(Note that this definition is consistent with that of \S\ref{sssec:flux}, if we view our row as the top row of a pipe dream.)
The leftmost vertical edge has flux $(XY)_{11}=t_1$, which we interpret as saying that the pipe numbered $1$ will enter
our row from the West end. 
Similarly, the rightmost vertical edge has flux $0$, which is interpreted as the lack of pipe exiting from the East end.

On the other hand, fluxes on vertical edges $0<j<n$, as well at the
bottom horizontal edges, are {\em a priori}\/ unconstrained: they
correspond to internal edges of a pipe dream, whose value will be
determined as we degenerate and pipe dreams start coalescing row by row.

As already observed in \S\ref{sssec:flux}, at each column $j$, the following ``flux conservation'' holds,
\begin{equation}\label{eq:conserv}
\begin{tikzpicture}[baseline=-3pt]
  {\setlength{\loopcellsize}{1.25cm}\plaq{}}
\node[shape=isosceles triangle,shape border rotate=90,inner sep=2pt,fill=blue,label={above:$\Phi_{\dN}$}] at (plaq.north) {};
\node[shape=isosceles triangle,shape border rotate=90,inner sep=2pt,fill=blue,label={below:$\Phi_{\dS}$}] at (plaq.south) {};
\node[shape=isosceles triangle,shape border rotate=0,inner sep=2pt,fill=blue,label={right:$\Phi_{\dE}$}] at (plaq.east) {};
\node[shape=isosceles triangle,shape border rotate=0,inner sep=2pt,fill=blue,label={left:$\Phi_{\dW}$}] at (plaq.west) {};
\end{tikzpicture}
\qquad
\Phi_\dW+\Phi_\dS=\Phi_\dE+\Phi_\dN
\end{equation}
where the arrows indicate in which direction the flux ``flows''.
This follows from $\Phi_\dN-\Phi_\dS=\Phi_\dW-\Phi_\dE=X_{1j}Y_{j1}$. 
A special case, which will appear repeatedly below, is the implication
(actually an equivalence)
\begin{equation}\label{eq:conservspec}
  X_{1j}=0\text{ or }Y_{j1}=0\quad\Longrightarrow\quad
  \Phi_\dW=\Phi_\dE\text{ and }\Phi_\dN=\Phi_\dS
\end{equation}
Using this language, \eqref{eq:degenflux} becomes
\begin{equation}\label{eq:degenflux2}
  X_{1j}(\Phi_\dS-\Phi_\dE)=0
\end{equation}
where it is understood that the fluxes refer to the square on column $j$.
There is no corresponding equation involving $Y_{j1}$, which will be the
source of much difficulty in lemma \ref{lem:Deqs} to come.

We now consider a particular component $E_\pi$ and degenerate it.

\subsection{$\tp$-degenerating the equations of $E_\pi$}
\subsubsection{The flux connectivity equations}
Consider the equations of proposition~\ref{prop:eqs} \textit{(ii)}, which were written more explicitly in \eqref{eq:flux}.
None of these are affected by the degeneration. In the first case,
if $i>1$, we can also write
\begin{equation}\label{eq:flux'}
  (YX)_{jj} = (X'Y')_{i-1\,i-1}\qquad j=\pi(i),\ i>1
\end{equation}
We recognize in the l.h.s.\ the flux at the top of column $j$.
If $i=1$, the same equation says that the leftmost vertical flux is equal to the flux at the top of column $j=\pi(1)$.

\subsubsection{A rank condition on $X$}
Next, among the equations of $E_\pi$, there are rank conditions on submatrices of $X$, cf proposition~\ref{prop:eqs} \textit{(iii)}.
An easy one comes from the $1\times (\pi(1)-1)$ upper-left submatrix of $X$, whose rank is zero.
These equations are unchanged by the degeneration, so we have
\begin{equation}\label{eq:Xzero}
  X_{1j}=0\qquad j=1,\ldots,\pi(1)-1
\end{equation}

\subsubsection{Some rank conditions on $Y$}
Similarly, among the equations of $E_\pi$, there are rank conditions
on submatrices of $Y$, cf proposition~\ref{prop:eqs} \textit{(iv)}.
We are interested in the cases $i=1$ and $j$ not in the image of $\pi$
(i.e., there will be no pipe coming out on the North side in column
$j$). Let $k=\#\{h \mid \pi(h)>j\}$. Then the rank of the submatrix of
$Y$ obtained by keeping only rows $\ge j$ must be less or equal to
$k$. Consider a $(k+1)\times(k+1)$ minor $\det Y_{JI}$ of $Y$ containing
the first column and whose topmost row is $j$,
so $I=\{1,i_1,\ldots,i_k\}$ and $J=\{j,j_1,\ldots,j_k\}$.
From the weight $[Y_{ij}]$ above, its degeneration is simply
\begin{equation}\label{eq:degenrank}
  Y_{j1} \det Y'_{J'I'} = 0
\end{equation}
with obvious notations $I'=\{i_1-1,\ldots,i_k-1\}$ and $J'=\{j_1,\ldots,j_k\}$.

\newcommand\pir{{\pi\backslash r}}
\subsection{Building the type $\tp$ row of tiles}\label{ssec:building}
We're now ready to degenerate $E_\pi$. Let us call its flat limit $D_\pi$.
Because $E_\pi$ is irreducible, all the geometric components of $D_\pi$
are of dimension $\dim E_\pi = m(n+1)$. We will first show that
\begin{equation}\label{eq:degen}
  D_\pi = \bigcup_{r} D_{\pi}^r \qquad \text{as a set; maybe not as a scheme}
\end{equation}
where $r$ runs overs some set of rows of tiles of type $\tp$ to be
described below, and $D_\pi^r$ is a variety whose defining equations
will be determined explicitly in terms of $r$.
Later we will show that $D_\pi$ is generically reduced.

\subsubsection{Square by square procedure}\label{ssec:square}
We build the row square by square. At first this is what the row might
look like, if say $\pi=352$:
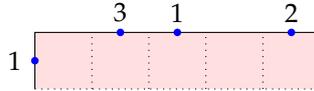
\[
\begin{tikzpicture}[scale=.75,every node/.style={circle,inner sep=1pt,fill=\linkpatternedgecolor}]
\fill[\linkpatternboxcolor] (0,0) rectangle (5,1);
\draw (0,0) -- (0,1) -- (5,1) -- (5,0);
\draw[dotted] (0,0) -- (5,0);
\foreach\x in {1,...,4} \draw[dotted] (\x,0) -- (\x,1);
\node[label=left:1] at (0,.5) {};
\node[label=above:3] at (1.5,1) {};
\node[label=above:1] at (2.5,1) {};
\node[label=above:2] at (4.5,1) {};
\end{tikzpicture}
\]
This picture should be interpreted as follows: the edges with solid lines
are those for which we know the value of the flux going through that edge,
with the convention that an empty edge (but not a dotted edge)
corresponds to a zero flux, and a edge with a pipe numbered $i$ going
through corresponds to the flux $t_i=(XY)_{ii}$.
Based on \eqref{eq:flux} and the definition of fluxes in \S\ref{ssec:defflux},
we know that the locations where pipes leave on the North side match the data of $\pi$
and the pipe numbered $1$ enters from the West (and nothing leaves from the East).
As we discover new equations satisfied by the various components of the degeneration of $D_\pi$, we add more tiles to the row.

\rem[gray]{unknown tiles shouldn't be confused with blank tiles, usual problem... dotted edges are there for this purpose}

First we take care of the columns $j<\pi(1)$. From \eqref{eq:Xzero}, combined with \eqref{eq:conservspec},
we conclude that fluxes are carried straight across every square in columns $j<\pi(1)$.
We therefore draw the appropriate ``straight'' tiles; e.g., on our running example,
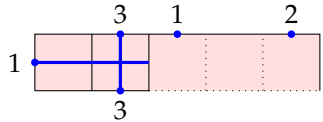
\[
\begin{tikzpicture}[scale=.75,every node/.style={circle,inner sep=1pt,fill=\linkpatternedgecolor}]
\fill[\linkpatternboxcolor] (0,0) rectangle (5,1);
\draw (0,0) -- (0,1) -- (5,1) -- (5,0);
\draw (0,0) -- (2,0) -- (2,1);
\draw (1,0) -- (1,1);
\draw[dotted] (2,0) -- (5,0);
\foreach\x in {3,4} \draw[dotted] (\x,0) -- (\x,1);
\draw[/linkpattern/edge] (0,0.5) -- (2,0.5)
(1.5,1) -- (1.5,0);
\node[label=left:1] at (0,.5) {};
\node[label=above:3] at (1.5,1) {};
\node[label=below:3] at (1.5,0) {};
\node[label=above:1] at (2.5,1) {};
\node[label=above:2] at (4.5,1) {};
\end{tikzpicture}
\]

Define $\tilde D_\pi^{\begin{tikzpicture}[scale=.25]
\fill[\linkpatternboxcolor] (0,0) rectangle (5,1);\useasboundingbox;
\draw (0,0) -- (0,1) -- (5,1) -- (5,0);
\draw (0,0) -- (2,0) -- (2,1);
\draw (1,0) -- (1,1);
\draw[dotted] (2,0) -- (5,0);
\foreach\x in {3,4} \draw[dotted] (\x,0) -- (\x,1);
\draw[/linkpattern/edge] (0,0.5) -- (2,0.5)
(1.5,1) -- (1.5,0);
\node at (0,.5) {};
\node at (1.5,1) {};
\node at (1.5,0) {};
\node at (2.5,1) {};
\node at (4.5,1) {};
\end{tikzpicture}
} := D_\pi$, the base case of an inductively defined family
of subschemes $\tilde D^r_\pi$.
We then fill in the remaining tiles starting from the {\em East.}
Below we will define $\tilde D_\pi^r$ as the corresponding piece of $D_\pi$,
where $r$ is the partially filled row; for example, after one step,
the example above might look like
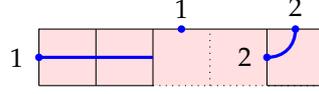
\[
\begin{tikzpicture}[scale=.75,every node/.style={circle,inner sep=1pt,fill=\linkpatternedgecolor}]
\fill[\linkpatternboxcolor] (0,0) rectangle (5,1);
\draw (0,0) -- (0,1) -- (5,1) -- (5,0) -- (4,0);
\draw (0,0) -- (2,0) -- (2,1);
\draw (1,0) -- (1,1);
\draw[dotted] (2,0) -- (5,0);
\draw[dotted] (3,0) -- (3,1);
\draw (4,0) -- (4,1);
\draw[/linkpattern/edge] (0,0.5) node[label=left:1] {} -- (2,0.5)
(4.5,1) node[label=above:2] {} .. controls (4.5,.7) and (4.3,.5) .. (4,.5) node[label=left:2] {};
(1.5,1) node[label=above:3] {} -- (1.5,0) node[label=below:3] {};
\node[label=above:1] at (2.5,1) {};
\end{tikzpicture}
\]

Note that at each step of the process, the state of the North and East edges of the rightmost unknown square (which is located in column $j>\pi(1)$) is known (i.e., the value of the flux through those edges is known to be either $0$ for a blank edge or $t_i$ if pipe $i$ goes through).
\footnote{\label{ft:nonred} A small caveat is that the values of the
  fluxes are only known set-theoretically, not scheme-theoretically;
  see the discussion of blank tiles. This issue is not relevant to the
  construction of $\tilde D_\pi^r$ and will eventually be addressed
  when the row is completed.}  There are two cases:
\begin{itemize}
\item Either at least one of these two edges (the North and East of
  the rightmost unknown square) has nonzero flux, i.e. a pipe is going through.
  Apply \eqref{eq:degenflux} (or equivalently
  \eqref{eq:degenflux2}): it shows that $\tilde D_\pi^r$ can be
  written as
\begin{equation}\label{eq:division}
  \tilde D_\pi^r= (\tilde D_\pi^r \cap \{X_{1j}=0\}) \cup (\tilde D_\pi^r \cap \{\Phi_\dS=\Phi_\dE\})
\end{equation}

Call the first piece $\tilde D_{\pi}^{r_1}$ where $r_1$ is $r$ on which the square in column $j$ has been filled with the appropriate straight tile \plaqctr{c}, \plaqctr{h} or \plaqctr{v}, as follows once again from\eqref{eq:conservspec}.

Call the second piece $\tilde D_{\pi}^{r_2}$ where $r_2$ is $r$ on which the square in column $j$ has been filled with the appropriate elbow tile \plaqctr{a}, \plaqctr{j} or \plaqctr{r}, since $\Phi_\dS=\Phi_\dE$ (and $\Phi_\dN=\Phi_\dW$).

\item Both edges are empty.  We note that \eqref{eq:degenflux} can be
  rewritten, using $\Phi_{\dN}=\Phi_{\dE}=0$, as $X_{1j}^2 Y_{j1}=0$.  We
  would like to show that $Y_{j1}=0$, but we'll postpone this till the
  row is completed.  For now, and ignoring nonreducedness issues (see
  footnote \ref{ft:nonred}), $X_{1j}Y_{j1}=0$ is enough to conclude
  that fluxes $\Phi_{\dS}$ and $\Phi_{\dW}$ are also zero, cf
  \eqref{eq:conservspec}; so we introduce $r'$ which is $r$ with a
  blank tile \plaqctr{} added on column $j$, and define
  $\tilde D_{\pi}^{r'} := \tilde D_\pi^r$.
\end{itemize}

In both cases, subtract one from $j$ and proceed with each piece as long as $j>\pi(1)$.

Finally, when we reach column $j=\pi(1)$,
we insert in $r$ an elbow tile as appropriate (thus satisfying \eqref{eq:conserv}), without changing $\tilde D_\pi^r$.

\subsubsection{Description of the corresponding schemes}
At the end of this process we have a union $D_\pi = \bigcup_r \tilde D_\pi^r $ where $r$ runs over some set
of completed rows of tiles, e.g., on the running example, $r$ might be
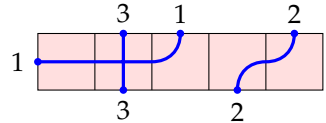
\[
\begin{tikzpicture}[scale=.75,every node/.style={circle,inner sep=1pt,fill=\linkpatternedgecolor}]
\draw[fill=\linkpatternboxcolor] (0,0) rectangle (5,1);
\foreach\x in {1,...,4} \draw (\x,0) -- (\x,1);
\draw[/linkpattern/edge] (0,0.5) -- (2,0.5)
 .. controls (2.3,0.5) and (2.5,0.7) .. (2.5,1)
(4.5,1) .. controls (4.5,.7) and (4.3,.5) .. (4,.5) .. controls (3.7,.5) and (3.5,.3) .. (3.5,0) node[label=below:2] {}
(1.5,1) -- (1.5,0);
\node[label=left:1] at (0,.5) {};
\node[label=above:3] at (1.5,1) {};
\node[label=below:3] at (1.5,0) {};
\node[label=above:1] at (2.5,1) {};
\node[label=above:2] at (4.5,1) {};
\end{tikzpicture}
\]

Given $\pi\in \SS_{m,n}$ and $r$ a row of tiles of type $\tp$ as above, we say that $r$ is \textbf{compatible} with $\pi$
if the pipes emerging at the top of $r$ are in columns $\pi(1),\ldots,\pi(m)$. By definition,
any $r$ that appears in the process described in \S\ref{ssec:square} is compatible with $\pi$;
it would not be hard to prove at this stage that conversely,
any $r$ that is compatible with $\pi$ appears this way, but instead we shall postpone this proof to \S\ref{ssec:degenmdeg}.

Given $\pi\in \SS_{m,n}$ and $r$ a compatible row, define
$\pir\in\SS_{m-1,n}$ to be the connectivity ``after removing the row $r$ from $\pi$''. More explicitly,
we label the pipes at the top of $r$ according to $\pi$; then follow these pipes across $r$,
and define $(\pir)(i)$ to be the column where pipe numbered
$i+1$ exits at the bottom of $r$
(the shift by $1$ is due to pipe $1$ having exited to the West).
On the example above, $\pir=42$.

Denote by $p$ the projection $(X,Y)\mapsto (X',Y')$. Recall that $E'$ is the lower-upper scheme in size $m-1$;
we introduce similar notations $E'_{\pi}$, $\tilde E'_{\pi}$ with $\pi\in \SS_{m-1,n}$.
\begin{lem}\label{lem:tech}\ %
  \begin{itemize}
    \item
  $p(\tilde D_\pi^r)$ is set-theoretically contained in $\tilde E'_{\pir}$.
\item The fibers of $\tilde D_\pi^r\overset{p}\to \tilde E'_{\pir}$ are of dimension at most $n+1$.
\end{itemize}
\end{lem}
\begin{proof}
  If the pipe numbered $i+1$ leaves $r$ on column $j$, this means the corresponding flux
$(Y'X')_{jj}$ is equal to $(XY)_{i+1\,i+1}=(X'Y')_{ii}$.
  The way we have defined $\pir$ is precisely so that by definition of $\tilde E'_{\pir}$
  (cf lemma~\ref{lem:flux}), one has
  $p(\tilde D_\pi^r)\subseteq \tilde E'_{\pir}$.

  For the second part, we first determine some equations satisfied by the fiber $p^{-1}(X',Y')$.
  Note that for each $j\ne \pi(1)$, by using \eqref{eq:flux}--\eqref{eq:flux'}
  we have an equation of the form 
  \begin{equation}\label{eq:fluxeq}
    X_{1j}Y_{j1}=\Phi_\dN-\Phi_\dS=
    \begin{cases}
      (X'Y')_{i-1\,i-1}-(Y'X')_{jj}&\text{if }\pi(i)=j\text{ for some }i
      \\
      -(Y'X')_{jj}&\text{else}
    \end{cases}    
  \end{equation}
  so the fiber sits inside a product of $n-1$ hyperbolas (or unions of
  two lines) and $\CC^2$. Therefore its dimension is at most $n+1$.
\end{proof}

At this stage, it is convenient to get rid of potential unwanted lower-dimensional pieces by defining
\[
  D_\pi^r := \overline{\tilde D_\pi^r \cap p^{-1}(E_\pir^{\prime\circ})}
\]
and giving it the reduced scheme structure.
Recall that $\tilde E'_{\pir}=E^{\prime\circ}_{\pir}\cup \text{lower dimensional}$
(cf lemmas~\ref{lem:dense} and \ref{lem:flux}),
and so by lemma~\ref{lem:tech},
$\tilde D_\pi^r \backslash D_\pi^r$ is of dimension $<\dim E'_{\pir}+\dim(\text{fiber})\le m(n+1)=\dim D_\pi$, so cannot contribute to $D_\pi$,
whose geometric components are all of the same dimension.
Hence $D_\pi$ is set-theoretically equal to the union  $\Union_r D^r_{\pi}$, as advertised in \eqref{eq:degen}.

We know quite a few equations of $D_\pi^r$ (and we will eventually
prove that these are {\em all}\/ the equations defining $D^r_\pi$):
\begin{lem}\label{lem:Deqs}
  The following equations hold (scheme-theoretically) on $\tilde D_\pi^r \cap p^{-1}(E_\pir^{\prime\circ})$,
  and therefore on $D_\pi^r\ni (X,Y)$:
  \begin{itemize}
  \item $(X',Y')\in E'_{\pir}$
  \item $X_{1j}Y_{j1}=\begin{cases}
    (X'Y')_{i-1\,i-1}-(Y'X')_{jj}
    \\[3mm]
    (X'Y')_{i-1\,i-1}
    \\[3mm]
    -(Y'X')_{jj}
  \end{cases}$ if there is an elbow tile in column $j$ of $r$
  $\begin{cases}
    \plaqctr{a}\qquad \pi(i)=j,\ i>1
    \\[2mm]
    \plaqctr{j}\qquad \pi(i)=j,\ i>1
    \\[2mm]
    \plaqctr{r}
  \end{cases}$
  \item $X_{1j}=0$ if there is a straight tile in column $j$ of $r$.
  \item $Y_{j1}=0$ if there is a blank tile in column $j$ of $r$.
  \end{itemize}
\end{lem}
\begin{proof}
  The first equation is just a reformulation of $p(X,Y)\in E^{\prime\circ}_{\pir}\subseteq E'_{\pir}$,
  and the next two are equations already satisfied on $\tilde D_\pi^r$, cf \eqref{eq:division} and \eqref{eq:fluxeq}.
  Only the last one is new.
  
  Pick $j$ such that $r$ has a blank tile in column $j$. We want to apply \eqref{eq:degenrank}.
  Because of this blank tile, if $i$ is such that $\pi(i)>j$ then we also know
  $(\pir)(i-1)>j$ (``any pipe that exits right of $j$ in $\pi$
  must also exit right of $j$ in $\pir$''). Let $k$ be the number of such indices $i$,
  and denote those indices $i_1<\cdots<i_k$.
  Also write $(\pir)(i_a-1)=j_a$ for $a=1,\ldots,k$.

  Now compute $\det Y'_{J'I'}$ as in \eqref{eq:degenrank} at the point $(X',Y')=(t (\pir),(\pir)^T s)$
  with $s$ and $t$ invertible diagonal matrices.
  $Y'_{J'I'}$ has exactly one nonzero entry per row and column, so its determinant is nonzero.
  We conclude that $Y_{j1}$ vanishes on $\tilde D_\pi^r \cap p^{-1}(E'{}^\circ_{\pir})$.
\rem[gray]{there might be a simpler proof, since we know $X_{1j}^2Y_{j1}=0$ -- but that's a bit tricky, because
we only derived this under the assumption that $\Phi_{\dN}=\Phi_{\dE}=0$, which itself was only known set-theoretically
-- so one would have to go through blanks one by one, and recompute the fluxes scheme-theoretically. I mean
for $D_\pi$ there are no such issues but we really want on $\tilde D_\pi \cap p^{-1}$ for the next corollary}
\end{proof}

Let $\hat D_\pi^r$ be the affine scheme cut out by the equations of lemma~\ref{lem:Deqs}.

\begin{lem}\label{lem:Dred}
  \begin{enumerate}
  \item $\hat D_\pi^r$ is reduced, and equidimensional of dimension $\dim D_\pi$.
  \item $D_\pi$ is generically reduced.
  \end{enumerate}
\end{lem}

\begin{proof}
  First, we prove that $\hat D_\pi$ is reduced, and of the expected
  dimension $\dim E'_\pir + (2n)-(n-1)=m(n+1)=\dim D_\pi$ (starting
  from $E'_\pir$, add $2n$ variables and impose $n-1$ equations).
  Recall that by definition, $E'_\pir$ is reduced (and irreducible).
  Looking at the equations of $\hat D_\pi$, and discarding the zero
  and free variables, we are left with an inductive application of the
  easy lemma that if $\Spec A$ is reduced and equidimensional and
  $a\in A$, then $\Spec(A[u,v]/(uv-a))$ is reduced and equidimensional of
  one higher dimension. (Proof: filter by degree in $u,v$ and take
  associated graded to reduce to the case $a=0$.)

  Next we address generic reducedness of $D_\pi$.
  From the discussion that follows lemma~\ref{lem:tech}, we only need to test smoothness at points
  in $\tilde D_\pi^r \cap p^{-1}(E_\pir^{\prime\circ})$ for some $r$.
  (In fact, $r$ is uniquely determined by $\pi$ and $\pir$,
  so $D_\pi \cap p^{-1}(E_\pir^{\prime\circ}) = \tilde D_\pi^r \cap p^{-1}(E_\pir^{\prime\circ})$.)%
  \junk{proof? not that we need.}
  Any component of $D_\pi^r$ with the expected dimension $\dim D_\pi$ is also a component of $\hat D_\pi^r$,
  so generic reducedness of $D_\pi$ follows from that of $\hat D_\pi^r$.
\end{proof}

\subsection{Type $\tp$ degeneration and $H_T^*$ classes}\label{ssec:degenmdeg}

Let us start from \eqref{eq:degen}. Thanks to lemma~\ref{lem:Dred},
$D_\pi$ can only differ from the union of $D_\pi^r$ by embedded components (which we conjecture
aren't there):
\begin{equation}\label{eq:degenx}
D_\pi = \bigcup_{\substack{r\text{ row of type }\tp\\\text{compatible with }\pi}} D_\pi^r  \qquad \text{(plus possible embedded components)}
\end{equation}
where conventionally, $D_\pi^r := \varnothing$ if $r$ does not appear
in the process described in \S\ref{ssec:square} (we shall show shortly
that this never occurs).

We want to show that the defining equations of $D_\pi^r$ are those given in lemma~\ref{lem:Deqs}, i.e.,
that $D_\pi^r=\hat D_\pi^r$. For now we only have inclusion, so
\begin{equation}\label{eq:degena}
D_\pi \subseteq \bigcup_{\substack{r\text{ row of type }\tp\\\text{compatible with }\pi}} \hat D_\pi^r  \qquad \text{(plus possible embedded components)}
\end{equation}

We now compute $H_T^*$ classes and conclude using the positivity
of our multigrading (see \S\ref{ssec:pos}) that the inclusion of \eqref{eq:degena} is an equality.

In order to proceed, we could compute $[\hat D_\pi^r]$.
Instead of doing so, we shall lazily bound it from above as follows:
by lexing variables say $X_{1i}$,
one can effectively set to zero the r.h.s.\ of the second set of equations in lemma~\ref{lem:Deqs}
(see lemma~\ref{lem:Dred} for a similar argument).
The resulting flat limit of $\hat D_\pi^r$ then sits inside the product of $E'_{\pir}$ with a complete intersection
inside $\CC^{2n}$ given by the $n-1$ equations, each involving $X_{1j}$ and $Y_{j1}$, $j\ne \pi(1)$
(in particular this product has the expected dimension $\dim D_\pi$).
Furthermore, computing the weights of these equations, we recognize
the contribution of the row of tiles $r$ (except for the tile in column $\pi(1)$), so that
\begin{equation}\label{eq:mdeg2}
  [\hat D_\pi^r] \le [E'_{\pir}] \prod_{\substack{j=1\\j\ne\pi(1)}}^n \wt_{\tp}(r_j,x_1-y_j)
\end{equation}
where $r_j$ is the the tile in the $j^{\rm th}$ column of $r$, and we recall from \S\ref{sec:intro} that
\[
  \wt_{\tp}(\text{tile},x_1-y_j)=
  \begin{cases}
    B-x_1+y_j&\text{blank tile}\\
    A+x_1-y_j&\text{straight tile}\\
    A+B&\text{elbow tile}
  \end{cases}
\]
It is understood that here,
$[E'_{\pir}]\in \ZZ[x_2,\ldots,x_m,y_1,\ldots,y_n,A,B]$ (i.e. no $x_1$)
since the $T$-weights of $(X',Y')$ live in that subring.

\eqref{eq:degena} implies
\begin{equation}\label{eq:mdeg1}
  [E_\pi] = [D_\pi] \le \Bigg[\bigcup_{\substack{r\text{ row of type }\tp\\\text{compatible with }\pi
    }} \hat D_\pi^r\Bigg] 
    \le \sum_{\substack{r\text{ row of type }\tp\\\text{compatible with }\pi
      }} [\hat D_\pi^r]
\end{equation}

Combining \eqref{eq:mdeg2} and \eqref{eq:mdeg1},
\begin{equation}\label{eq:mdeg3}
  [E_\pi] 
\le \sum_{\substack{r\text{ row of type }\tp\\\text{compatible with }\pi}} [E'_{\pir}] \prod_{\substack{j=1\\j\ne\pi(1)}}^n \wt_{\tp}(r_j,x_1-y_j)
\end{equation}

We can now apply corollary~\ref{cor:main}
first to $E_\pi$, where $\beta=(\tp,\beta_2,\ldots,\beta_m)$ is any binary sequence which starts with $\tp$:
\begin{align*}
  [E_\pi] &=(A+B)^{-m}G_\pi = (A+B)^{-m}\hspace{-5mm}\sum_{\substack{m\times n\text{ GPDs of type }\beta\\\text{compatible with }\pi}}
  \prod_{i=1}^m \prod_{j=1}^n \wt_{\beta_i}(\text{tile}(i,j),x_i-y_j)
  \\
              &= \sum_{\substack{r\text{ row of type }\tp\\\text{compatible with }\pi}} \prod_{\substack{j=1\\j\ne\pi(1)}}^n \wt_{\tp}(r_j,x_1-y_j)\ (A+B)^{-m+1}\hspace{-10mm} \sum_{\substack{(m-1)\times n\text{ GPDs of type }\beta'\\\text{and connectivity }\pir}}
\prod_{i=1}^{m-1} \prod_{j=1}^n \wt_{\beta_{i+1}}(\text{tile}(i,j),x_{i+1}-y_j)
  \\
  &=\sum_{\substack{r\text{ row of type }\tp\\\text{compatible with }\pi}} [E'_{\pir}] \prod_{\substack{j=1\\j\ne\pi(1)}}^n \wt_{\tp}(r_j,x_1-y_j)
\end{align*}
where $\beta'=(\beta_2,\ldots,\beta_m)$, and in the last line we've applied once again corollary~\ref{cor:main}, but this time to $E'_{\pir}$.
In conclusion, the inequality of \eqref{eq:mdeg3} is an equality,
hence so are all intermediate inequalities. 
In particular, 
the $\hat D_\pi^r$ intersect in pairs in codimension $>0$,
and $D_\pi^r = \hat D_\pi^r$ as sets,
but since both are reduced (lemma~\ref{lem:Dred})
we have $D_\pi^r = \hat D_\pi^r$.

We are ready to conclude:
\begin{prop}\label{prop:degen}
  The flat limit $D_\pi$ of $E_\pi$ takes the form
  \[
    D_\pi = \bigcup_{\substack{r\text{ row of type }\tp\\\text{compatible with }\pi}} D_\pi^r
    \quad \text{(plus possible embedded components)}
  \]
  where the $D_\pi^r$ are the geometric irreducible components of $D_\pi$ and are given by the equations of lemma~\ref{lem:Deqs}.
\end{prop}
\begin{proof}
    The only thing left to prove is that $D_\pi^r$ is irreducible.
  If $(X,Y)\in \tilde D_\pi^r \cap p^{-1}(E'{}^\circ_{\pir})$, then the equations of lemma~\ref{lem:Deqs} satisfied by the $X_{1j}$ and $Y_{j1}$
  take the form: $X_{1j}=0$ or $Y_{j1}=0$ or $X_{1j}Y_{j1}=\,$(something nonzero). Therefore
  $\tilde D_\pi^r \cap p^{-1}(E'{}^\circ_{\pir})$ is a fiber bundle over the variety $E'{}^\circ_{\pir}$ with irreducible fibers,
  hence irreducible, and so is its closure $D_\pi^r$.
\end{proof}

\subsection{Type $\bt$ degeneration}\label{ssec:bdegen}
Recall the involution $\iota$ on $E$ from \S\ref{ssec:invo}.
We define the type $\bt$ one-row degeneration of $E$ and its components by 
conjugating the type $\tp$ one-row degeneration with $\iota$.
It corresponds to degenerating the last row (resp.\ column) of $X$ (resp.\ $Y$). 

In practice, it will be useful to write some of the formul\ae\ explicitly.

The weights of the circle action are
\[
[X_{ij}]=\begin{cases}
  j&i=m
\\
0&\text{else}
\end{cases}
\qquad
[Y_{ji}]=\begin{cases}
-j&i=m
\\
0&\text{else}
\end{cases}
\]
\rem[gray]{same as $\tp$-degeneration -- there used to be $j-(n+1)$ and $n+1-j$ but it makes no difference!
  i.e., one can lex/revlex from one end or revlex/lex from the other end of the row, same thing}
so we are degenerating the last row (resp.\ column) of $X$ (resp.\ $Y$).

The fluxes on a row of type $\bt$ are defined as follows:
\begin{itemize}
\item On the top horizontal edge on column $j$, define the flux to be $(YX)_{jj}$.
\item On the bottom horizontal edge on column $j$, define the flux to be $(Y'X')_{jj}$.
\item On the vertical edges, numbered from $0$ to $j$, define the flux to be
  $\sum_{j'<j} X_{mj'}Y_{j'm}$.
\end{itemize}
One has $\Phi_\dN-\Phi_\dS=X_{mj}Y_{jm}=\Phi_\dE-\Phi_\dW$, so the flux conservation equation reads
\begin{equation}\label{eq:conservb}
\begin{tikzpicture}[baseline=-3pt]
  {\setlength{\loopcellsize}{1.25cm}\plaq{}}
\node[shape=isosceles triangle,shape border rotate=90,inner sep=2pt,fill=blue,label={above:$\Phi_\dN$}] at (plaq.north) {};
\node[shape=isosceles triangle,shape border rotate=90,inner sep=2pt,fill=blue,label={below:$\Phi_\dS$}] at (plaq.south) {};
\node[shape=isosceles triangle,shape border rotate=180,inner sep=2pt,fill=blue,label={right:$\Phi_\dE$}] at (plaq.east) {};
\node[shape=isosceles triangle,shape border rotate=180,inner sep=2pt,fill=blue,label={left:$\Phi_\dW$}] at (plaq.west) {};
\end{tikzpicture}
\qquad
\Phi_\dE+\Phi_\dS=\Phi_\dW+\Phi_\dN
\end{equation}

If $r$ is a row of type $\bt$ which is compatible with $\pi$,
define (similarly to type $\tp$)
$\pir\in\SS_{m-1,n}$ to be the connectivity after removing the row $r$ from $\pi$, except no renumbering is needed,
so $(\pir)(i)$ is the column where pipe numbered
$i$ exits at the bottom of $r$.

The analogue of lemma~\ref{lem:Deqs} and proposition~\ref{prop:degen} is
\begin{prop}\label{prop:degenb}
  The flat limit $F_\pi$ of $E_\pi$ under type $\bt$ degeneration has a decomposition into irreducible components
  \[
    F_\pi = \bigcup_{\substack{r\text{ row of type }\bt\\\text{compatible with }\pi}} F_\pi^r
        \quad \text{(plus possible embedded components)}
  \]
  where $F_\pi^r$ is given by the following equations:
  if $(X,Y)\in F_\pi^r$ and $X'=(X_{ij})_{\substack{1\le i\le m-1\\1\le j\le n}}$, $Y'=(Y_{ji})_{\substack{1\le j\le n\\1\le i\le m-1}}$,
  \begin{itemize}
  \item $(X',Y')\in E'_{\pir}$.
  \item $X_{mj}Y_{jm}=\begin{cases}
    (X'Y')_{ii}-(Y'X')_{jj}
    \\[3mm]
    (X'Y')_{ii}
    \\[3mm]
    -(Y'X')_{jj}
  \end{cases}$ if there is an elbow tile in column $j$ of $r$
  $\begin{cases}
    \plaqctr{b}\qquad \pi(i)=j,\ i<m
    \\[2mm]
    \plaqctr{i}\qquad \pi(i)=j,\ i<m
    \\[2mm]
    \plaqctr{k}
  \end{cases}$
  \item $Y_{jm}=0$ if there is a straight tile in column $j$ of $r$.
  \item $X_{mj}=0$ if there is a blank tile in column $j$ of $r$.
  \end{itemize}
  and its $H_T^*$ class is given by
  \[
    [F_\pi^r]
  =[E'_{\pir}] \prod_{\substack{j=1\\j\ne\pi(m)}}^n \wt_{\bt}(r_j,x_m-y_j)
  \]
where
\[
  \wt_{\bt}(\text{tile},x_m-y_j)=
  \begin{cases}
    A+x_m-y_j&\text{blank tile}\\
    B-x_m+y_j&\text{straight tile}\\
    A+B&\text{elbow tile}
  \end{cases}
\]
\end{prop}

\subsection{Putting it all together}
At this stage, all the ingredients are in place to build inductively the full degeneration of $E_\pi$:
we apply a one-row degeneration as in \ref{ssec:tdegen}--\ref{ssec:bdegen},
then notice that there are two types of equations of each piece $D_\pi^r$ (resp.\ $F_\pi^r$) in lemma~\ref{lem:Deqs}
(resp.\ proposition~\ref{prop:degenb}):
\begin{itemize}
\item Equations involving $X_{1j}$ and $Y_{j1}$ (resp.\ $X_{mj}$ and $Y_{jm}$) and fluxes.
The latter are linear combinations of $X'_{ij}Y'_{ji}$, which are unaffected by any of our degenerations, so we can safely ignore them.
\item Equations involving the remaining variables, of the form $(X',Y')\in E'_{\pir}$.
\end{itemize}
We can therefore apply one more degeneration to $E'_{\pir}$ and keep repeating until $m=0$.

By concatenating the $m$ rows produced this way from top to bottom, we obtain a full $m\times n$ pipe dream.
In order to write explicitly the equations of the resulting variety, 
recall the flux variables from \S\ref{sssec:flux}:  \rem[gray]{convention is opposite in m2 file: ``vertical flux'' is across horizontal edge}
\begin{alignat*}{2}
\Phi_{V(i,j)} &= \sum_{\substack{j'>j\text{ row }\tp\\j'\le j\text{ row }\bt}} X_{\varphi(i)j'}Y_{j'\varphi(i)}
\qquad
&&\begin{aligned} i&=1,\ldots,m\\ j&=0,\ldots,n
\end{aligned}
\\
  \Phi_{H(i,j)} &= \sum_{i'>i} X_{\varphi(i')j}Y_{j\varphi(i')}
  \\
  &= \sum_{i_1<i'<i_2} X_{i'j}Y_{ji'}
\qquad
&&\begin{aligned} i&=0,\ldots,m\\ j&=1,\ldots,n
\end{aligned}
\qquad
\begin{aligned}i_1&=\#\{a\le i\mid \beta_a=\tp\}\\ i_2&=m+1-\#\{a\le i\mid \beta_a=\bt\}
\end{aligned}
\end{alignat*}

As a consistency check,
the fluxes on boundary edges are $(YX)_{jj}$ on the North side column $j$, $(XY)_{\varphi(i)\varphi(i)}$ on West or East ends of row $i$
depending on whether $\beta_i=\tp$ or $\bt$, and zero elsewhere. 
Because flux equations are unaffected by degeneration,
\eqref{eq:flux} satisfied on $E_\pi$ is also satisfied on any component of the degeneration.
This means the pipe numbered $\varphi(i)$ exits at the top at column $j=\pi(\varphi(i))$,
i.e., the pipe dream
has connectivity $\pi$, as foreshadowed at the end of \S\ref{ssec:flux}.

Also, recall the notation $t_i:=(XY)_{ii}$ for the flux of the pipe labeled $i$.

We can at last state our main result, which is an expanded version of theorem~\ref{thm:main} \textit{(2)} and proposition~\ref{prop:flux}:

\begin{thm}\label{thm:mainfull}
  For each hybridization $\beta$, there is a (partial Gr\"obner)
  degeneration of each $E_\pi$ into a union of varieties $V_\delta$
  indexed by hybrid generic pipe dreams $\delta$ of type $\beta$ with
  connectivity $\pi$ (possibly, though conjecturally not, with the
  addition of some embedded components), where $V_\delta$ is given by
  \[
    V_\delta := \left\{ (X,Y)\in \Mat(m,n,\CC)\times \Mat(n,m,\CC)
      : 
      \begin{aligned}
        X_{\varphi(i)j}&=0\text{ if square $(i,j)$ is a }
        \left\{
          \begin{aligned}
            \text{straight tile and }\beta_i&=\tp
            \\
            \text{blank tile and }\beta_i&=\bt
          \end{aligned}\right.
        \\
        Y_{j\varphi(i)}&=0\text{ if square $(i,j)$ is a }
        \left\{
          \begin{aligned}
            \text{blank tile and }\beta_i&=\tp
            \\
            \text{straight tile and }\beta_i&=\bt
          \end{aligned}\right.
        \\
        \Phi_e&=\begin{cases}t_i&\text{if pipe $i$ passes through edge $e$}\\0&\text{else}\end{cases}
      \end{aligned}
    \right\}
  \]
  Each $V_\delta$ is a complete intersection, and the contribution
  of the pipe dream $\delta$ to $G_\pi$ is $(A+B)^m$ times $[V_\delta]$.
\end{thm}

\begin{proof}
As explained above, we apply repeatedly the degenerations of type $\beta_i$, $i=1,\ldots,m$.
The flux equations ($\Phi_e=t_i$ or $0$) are a consequence of the second equations in the list of lemma~\ref{lem:Deqs}
and proposition~\ref{prop:degenb}, as can be shown by an easy induction.

Note that as stated, we have more defining equations for $V_\delta$ than its codimension $m(n-1)$. The reason is that the flux
equations, viewed as linear equations on the $X_{ij}Y_{ji}$, are not linearly independent. 
One can work out by hand the combinatorics of independent flux equations, or use the existing counting of equations given
in lemma~\ref{lem:Deqs} and proposition~\ref{prop:degenb}, where it is shown that every row contributes $n-1$ equations,
which implies that $V_\delta$ is a complete intersection; among which
flux equations (of weight $A+B$) for each elbow tile except one. This leaves a discrepancy of $m$ elbow tiles that contribute
to $G_\pi$ but not to $[V_\delta]$, hence the factor of $(A+B)^m$. The weights of the $X_{\varphi(i)j}=0$ and $Y_{j\varphi(i}=0$ equations
match the weights of the corresponding tiles, cf \eqref{eq:wt}.

The irreducibility of $V_\delta$ follows from its construction as the closure of an iterated fiber bundle with irreducible fibers.%
\junk{is this good enough? AK:yes}
$[V_\delta]=\mu\, [\text{red}(V_\delta)]$ where $\mu\in\ZZ_{\ge1}$ is the multiplicity of $V_\delta$. Since no factor of $[V_\delta]$
is divisible by $\mu>1$, we learn $\mu=1$, i.e., $V_\delta$ is generically reduced, and since
it's a complete intersection, it's reduced, hence a variety.
\end{proof}

Finally, we reconnect the involutions $\iota$ (\S\ref{ssec:invo}) and left-right mirror image on pipe dreams (\S\ref{ssec:mirror}):
\begin{prop}
Let $\delta$ be a hybrid generic pipe dream of type $\beta$ with connectivity $\pi$,
and $\bar\delta$ its mirror image, 
a hybrid generic pipe dream of type $\bar\beta$ with connectivity $\gamma_n\pi\gamma_m$.
Then $\iota(V_\delta)=V_{\bar\delta}$.
\end{prop}

\begin{proof}[Sketch of proof]
  This is of course a consequence of the definition of type $\bt$
  degeneration as type $\tp$ degeneration conjugated by $\iota$, but
  one can also prove this directly.  As already pointed out in the
  proof of lemma~\ref{lem:mirror}, the numbering of pipes is reversed
  by the mirror image map, so $\bar\varphi=\gamma_m\varphi$. This
  combined with the left-right mirror image effectively implements a
  $180^\circ$ rotation of matrices $X$ and $Y$. The substitution
  $\beta\mapsto\bar\beta$ then exchanges the roles of $X$ and $Y$.  We
  leave it as an exercise to the reader to check in more detail that
  applying $\iota$ to the equations of $V_\delta$ given in
  theorem~\ref{thm:mainfull} produces those of $V_{\bar\delta}$.
\end{proof}

\appendix

\section{Equidimensionality of the rectangular lower-upper scheme}\label{app:equi}
In this appendix we show, as claimed in \S\ref{sec:lowerupper},
that the only irreducible components of $E$ are its top-dimensional
components $E_\pi$.
This result is similar in spirit to the main result of \cite{Rothbach}.

Given a $m\times n$ partial permutation matrix $\pi$, define a new (maximal rank)
partial permutation matrix $\bar\pi$
as follows: number the zero rows (resp.\ columns) of $\pi$, $i_1<\cdots<i_k$ (resp.\ $j_1<\cdots<j_{n-m+k}$). Then
\[
\bar\pi_{ij}=
\begin{cases}
1&\pi_{ij}=1\text{ or } (i=i_r\text{ and } j=j_{n-m+r})
\\
0&\text{else}
\end{cases}
\]
For example, if $\pi=\left(\begin{smallmatrix} 
0 & 0 & 0 & 0 & 0 & 0 \\
 0 & 0 & 1 & 0 & 0 & 0 \\
 1 & 0 & 0 & 0 & 0 & 0 \\
 0 & 0 & 0 & 0 & 0 & 0
\end{smallmatrix}\right)$,
then
$\bar\pi=\left(\begin{smallmatrix} 
 0 & 0 & 0 & 0 & 1 & 0 \\
 0 & 0 & 1 & 0 & 0 & 0 \\
 1 & 0 & 0 & 0 & 0 & 0 \\
 0 & 0 & 0 & 0 & 0 & 1
\end{smallmatrix}\right)$.

\begin{lem}\label{lem:equidense}
$
  E^\circ_\pi:=(N_-\times N_+)\cdot\left\{ (s\pi, \bar\pi^T t),\ s,t\ m\times m \text{ diagonal},\ st\text{ invertible with
    distinct eigenvalues}\right\}
$
is an open dense subset of $E^1_\pi$.
\end{lem}
\begin{proof}
We need to carefully generalize the proof of lemma~\ref{lem:dense}. We denote $I$ the complement of $\{i_1,\ldots,i_k\}$
in $\{1,\ldots,m\}$.

Computing the infinitesimal stabilizer of $(s\pi, \bar\pi^Tt)$ under $N_-\times N_+$ action leads to
\[
A_- s \pi - s \pi A_+ = 0,\qquad A_+ \bar\pi^T t - \bar\pi^T t A_- = 0
\]
where $A_-$ is $m\times m$ strictly lower triangular, and $A_+$ is $n\times n$ strictly upper triangular.

Pick $1\le i<i'\le m$. There are two possibilities:
\begin{itemize}
\item If $i\in I$ or $i'\in I$, we use the same argument as in the maximal rank case, namely, compute
\[
A_- D = D A_- \qquad D:=s\pi\bar\pi^Tt
\]
Now $D$ is a diagonal matrix with entries $(st)_{ii}$ for $i\in I$, zero otherwise.
By taking the entry $(i',i)$ of this equation, and using the fact that $D_{ii}\ne D_{i'i'}$
as long as at least one of $i$ or $i'$ is in $I$, we conclude that $(A_-)_{i'i}=0$.

\item Otherwise $i\not\in I$ and $i'\not\in I$.
Define $\bar\pi(i)=j$, $\bar\pi(i')=j'$: according to the definition of $\bar\pi$, $j<j'$.
Compute the entry $(j',i)$ of the second equation: 
\begin{align*}
(A_+\bar\pi^T t)_{j'i}&=(A_+)_{j'j} t_i=0\qquad j'>j
\\
(\bar\pi^T t A_-)_{j'i}&=t_{i'} (A_-)_{i'i}
\end{align*}
so $(A_-)_{i'i}=0$.
\end{itemize}
We conclude once again that $A_-=0$. We're left with the equations
\[
\pi A_+=0 \qquad A_+\bar\pi^T=0
\]
Clearly this implies that $(A_+)_{jj'}=0$ if either $j$ is in the image of $\pi$, or $j'$ in the image of $\bar\pi$.
Note however that if $j$ is in the image of $\bar\pi$, it is either in the image of $\pi$ or it's greater than any $j'$
not in the image of $\bar\pi$ (i.e., than any $j' \in\{j_1,\ldots,j_{n-m}\}$ by definition of $\bar\pi$), in which case $(A_+)_{jj'}=0$ by upper triangularity.
In conclusion, $(A_+)_{jj'}=0$ unless neither $j$ nor $j'$ are in the image of $\bar\pi$.
This leaves once again a strict $(n-m)\times(n-m)$ upper triangle of free entries.

The dimension of 
$(N_-\times N_+)\cdot\left\{ (s\pi, \bar\pi^T t)\right\}$ is the dimension of the group plus the dimension of the set
$\left\{ (s\pi, \bar\pi^T t)\right\}$ minus the dimension of the stabilizer,
that is
\[
m(m-1)/2+n(n-1)/2+\rank\pi+\rank \bar\pi-(n-m)(n-m-1)/2=mn+\rank\pi
\]
which equals the dimension of $E^1_\pi$, hence the density statement.
\end{proof}

\begin{cor}\label{cor:equiincl}
$E_\pi^1 \subseteq E_{\bar\pi}$.
\end{cor}
\begin{proof}
Given $(X,Y)\in E_\pi^\circ$, where without loss of generality one may assume $X=s\pi$, $Y=\bar\pi^Tt$,
consider $X_\alpha=s ((1-\alpha)\pi + \alpha \bar\pi)$, $\alpha\in\CC$.
One has $(X_0,Y)=(X,Y)$ and for generic $\alpha$, $(X_\alpha,Y)\in E^1_{\bar\pi}\subset E_{\bar\pi}$.
The inclusion follows by the density statement of lemma~\ref{lem:equidense}.
\end{proof}

\begin{thm}\label{thm:Eequidim}
  $E$ is equidimensional.
\end{thm}
\begin{proof}
$E=\bigcup_\pi \overline{E_\pi^1}$ decomposes into a union of finitely many irreducible closed subvarieties,
so these are the potential irreducible components of $E$.
But if $\pi$ is not maximal rank, one has according to corollary~\ref{cor:equiincl}
$\overline{E_\pi^1}\subset \overline{E_{\bar\pi}^1}$ with $\bar\pi\ne\pi$.
So the only components are the $E_\pi=\overline{E_{\pi}^1}$ with $\pi$ maximal rank,
i.e., the top-dimensional ones according to lemma~\ref{lem:dim}.
\end{proof}

\gdef\MRshorten#1 #2MRend{#1}%
\gdef\MRfirsttwo#1#2{\if#1M%
MR\else MR#1#2\fi}
\def\MRfix#1{\MRshorten\MRfirsttwo#1 MRend}
\renewcommand\MR[1]{\relax\ifhmode\unskip\spacefactor3000 \space\fi
\MRhref{\MRfix{#1}}{{\scriptsize \MRfix{#1}}}}
\renewcommand{\MRhref}[2]{%
\href{http://www.ams.org/mathscinet-getitem?mr=#1}{#2}}
\bibliographystyle{amsalphahyper}
\bibliography{biblio}
\end{document}